\documentclass[11pt]{amsart}

\usepackage{amssymb}
\usepackage{amsmath}
\usepackage{amsthm}
\usepackage{amscd}

\def\NN{{\Bbb N}}

\def\BB{{\Bbb B}}

\def\XX{{\Bbb X}}
\def\CC{{\Bbb C}}
\def\Xf{\mathcal{X}}
\def\Bf{\mathcal{B}}
\def\Df{\mathcal{D}}
\def\Pf{\mathcal{P}}
\def\Af{\mathcal{A}}
\def\Ef{\mathcal{E}}
\def\Ff{\mathcal{F}}
\def\Mf{\mathcal{M}}
\def\kap{\varkappa}
\def\TT{{\Bbb T}}
\def\Cs{{{C}}}
\renewcommand{\Re}{{\Bbb R}}
\def\eps{\varepsilon}

\def\1{1\!\!\hbox{{\rm I}}}

\def\eqdef{\mathop{=}\limits^{df}}
\def\ax{\Re^+}
\newcommand{\prt}{\partial}
\newcommand{\be}{\begin{equation}}
\newcommand{\ee}{\end{equation}}
\newcommand{\ba}{\begin{aligned}}
\newcommand{\ea}{\end{aligned}}

\theoremstyle{plain}
\newtheorem{thm}{Theorem}[section]
\newtheorem{lem}{Lemma}[section]
\newtheorem{prop}{Proposition}[section]
\newtheorem{cor}{Corollary}[section]
\theoremstyle{definition}
\newtheorem{dfn}{Definition}[section]
\newtheorem{rem}{Remark}[section]

\numberwithin{equation}{section}

\begin{document}

\title[Asymptotic and spectral properties ]
    {Asymptotic and spectral properties of
exponentially  $\phi$-ergodic Markov processes}

\author{Alexey M. Kulik}
\address{Kiev 01601 Tereshchenkivska str. 3, Institute of Mathematics,
Ukrai\-ni\-an National Academy of Sciences}
\email{kulik\@imath.kiev.ua}
\thanks{Research is partially supported by National Academy of Science of Ukraine, project \# 184 -- 2008}


\begin{abstract}
New relations between ergodic rate, $L_p$ convergence rates, and
asymptotic behavior of tail probabilities for hitting times of a
time homogeneous Markov process are established. For $L_p$
convergence rates and related spectral and functional properties
(spectral gap and Poincar\'e inequality) sufficient conditions are
given in the terms of an exponential $\phi$-coupling. This
provides sufficient conditions for $L_p$ convergence rates in the
terms of appropriate combination of `local mixing' and
`recurrence' conditions on the initial process, typical in the
ergodic theory of Markov processes. The range of application of
the approach includes time-irreversible processes. In particular,
sufficient conditions for spectral gap property for L\'evy driven
Ornstein-Uhlenbeck process are established.
\end{abstract}

\maketitle

\section{Introduction}

In this paper, we establish new relations between
three topics related to the asymptotic behavior of a time homogeneous
Markov process:

\begin{itemize}
\item  {\it ergodic rate}; that is,
the rate of convergence of the transition probabilities to the
invariant measure of the process;

\item \emph{$L_p$ convergence rates}; that is, rates of convergence for $L_p$-semigroups generated by the process;

\item tail probabilities for \emph{hitting times} of the process.
\end{itemize}

It is well known  that $L_p$ (especially, $L_2$) convergence
rates for a Markov process  are closely related with a number of
intrinsic functional features: the spectral gap property for the
generator of the process,  the Poincar\'e inequality for the
associated Dirichlet form,  Cheeger-type isoperimetric inequality
for the invariant measure. On the other hand, the classic methods
of the ergodic theory of Markov processes allow one to establish
ergodic rates under quite simple and transparent conditions on
the process that do not involve any essential limitation on the
structure of the state space. Our intent is to extend the range
of applications of these methods in order to provide similar
conditions for $L_p$ convergence rates.

It looks very unlikely that  $L_p$ convergence rates  can be deduced  from ergodic ones straightforwardly.
The ergodic rates are, in fact, norm estimates for a semigroup of operators in $\BB(\XX)$.
   In general, one  have no means to expect that such estimates  would produce a norm estimate for
    semigroup of operators in $L_p(\XX,\pi)$ with some measure $\pi$ when the state space $\XX$ is
    of a complicated structure. This guess is supported by concrete examples, see  section
    4
    below.

It is well known  that for a Markov process with a finite state
space three topics listed above are, in fact, equivalent; see the
detailed exposition in \cite{AF}, Chapters 2 -- 4. For a process
with at most countable state space,  relations between its ergodic
properties and rates of convergence for related  $L_2$-semigroup
were studied in  \cite{Chen00}. However,  the methods   of
\cite{Chen00} exploit the representation of the state space  as a
countable collection of points, and hardly admit a straightforward
generalization to a general case. In this paper, we propose a new
point of view.  Let us explain the main idea of our approach
briefly; a more detailed discussion is given in sections 2 and 3
below.

We start our considerations not from the estimate for the ergodic
 rate of a Markov process itself, but from the  auxiliary construction
 of a {\it coupling}, which is a standard tool for proving such an estimate.
 This construction appears to be an appropriate tool for getting  $L_p$ estimates
  as well, see section 3 below. In such a way, we are able to establish estimates
   for $L_p$  convergence rates  under the typical conditions used in the ergodic theory of Markov
   processes.

Usually, such conditions  include  some {\it local mixing
conditions}, and some {\it recurrence conditions}.  The former
ones are discussed in details  in section 2.1; the latter ones
can be formulated in the terms of {\it hitting times} of some
sets by the process $X$. Henceforth, in our  framework, estimates
for the hitting times  are involved, as sufficient conditions,
 both into  ergodic  rates and into $L_p$ convergence rates for the process. On the other hand,
  it is known (\cite{Mat97}) that the functional inequalities like the Poincar\'e one imply moment
  estimates for hitting times. Therefore
three topics listed at the beginning of the Introduction are  closely related indeed.
In fact, our approach allows us to give, for  some classes of the processes, necessary
and sufficient conditions that  describe relations between these topics completely.

The range of
    applications of our approach is not restricted to time-reversible processes. For time-reversible processes
    respective $L_2$-generators are self-adjoint  which makes possible to apply the spectral decomposition
theorem in order to get one-to-one correspondence between ergodic
rates and $L_2$-convergence rates; see \cite{RR97} and references
therein for discrete-time case and \cite{Chen00}, Theorem 1.2 for
continuous-time case. Our approach does not rely heavily  on the
spectral decomposition theorem. This makes possible to consider,
for instance, solutions to SDE's with jump noise which typically
are irreversible (in time).

The structure of the article is following. In section 2, we give
basic notions and constructions
 required for the main exposition. In particular, we introduce the notion of an \emph{exponential $\phi$-coupling},
  which is the main tool in our approach. Section 3  contains the main part of the paper devoted to the proof of
    $L_p$ convergence rates and related functional properties in the terms of the  exponential $\phi$-coupling property.
     In section 4 we consider one example of a Markov process and use it to demonstrate  main statements,
     as well as relations between the  \emph{exponential $\phi$-coupling}, \emph{growth bound},
      \emph{spectral gap}, and \emph{Poincar\'e inequality}. Section 5 contains an application of
       the main results to  L\'evy driven Ornstein-Uhlenbeck processes. In the recent paper \cite{Kul09}, ergodic rates for processes defined by  L\'evy driven SDE's are established.
        Here, we extend these results and describe spectral properties of a generator for some class of
         such processes. We have already mentioned that solutions to
L\'evy driven SDE's, typically, are irreversible (in time). Hence
the corresponding theory for $L_p$ semigroups appears to be
substantially more complicated than, for instance,  respective
theory for diffusion processes. For diffusion processes, we
establish in section 7 {\it a criterion} which gives one-to-one
correspondence between three topics mentioned at the beginning of
Introduction. This criterion extends, in particular, the
sufficient condition from \cite{RW04}, Theorem 1.1 for a diffusion
process to satisfy the Poincar\'e inequality. The proof of this
criterion is based on the main results from section 3 and
exponential integrability of the hitting times under the
Poincar\'e inequality. The latter statement is proved in section
6, and performs  an improvement of the integrability result from
\cite{Mat97}.

\section{Notation and basic constructions}

\subsection{Elements of ergodic theory for Markov processes} We consider a time homogeneous Markov process $X=\{X_t,
t\in\ax\}$ with a locally compact metric space $(\XX,\rho)$ as
the state space. The process $X$ is supposed to be strong Markov
and to have c\'adl\'ag trajectories. The transition function for
the process $X$ is denoted by $P_t(x,dy), t\in\ax, x\in\XX$. We
use standard notation $P_x$ for the distribution of the process
$X$ conditioned that $X_0=x$, and $E_x$ for the expectation w.r.t.
$P_x$ ($x\in\XX$ is arbitrary).

All the functions on $\XX$ considered in the paper are assumed to
be measurable w.r.t. Borel $\sigma$-algebra $\Bf(\XX)$. The set
of probability measures on $(\XX,\Bf(\XX))$ is denoted by
$\Pf(\XX)$. For a given $\mu\in\Pf(\XX)$ and $t\in\ax$, we denote
$ \mu_t(dy)\eqdef \int_\XX P_t(x,dy)\, \mu(dx). $ Clearly, $\mu_t$
coincides  with the distribution of the value $X_t$  assuming
that the distribution of the initial value $X_0$ equals $\mu$.
Probability  measure $\mu$ is called an invariant measure for $X$
if  $\mu_t=\mu, t\in\ax$.

In our considerations, we are mostly interested in the processes
on a non-compact state spaces, such as  diffusions on non-compact
manifolds or  solutions to SDE's with a jump noise. Typically,
for such a processes there does not exist a uniform (in $\mu$)
estimate for the rate of convergence rate of convergence of
$\mu_t$ to $\mu$ w.r.t. to the total variation distance. For such
a processes, the notion of $(r,\phi)$-ergodicity appears to be
most natural (see \cite{DFG09} and discussion therein). Let us
expose this notion and related objects.

Let $\phi:\XX\to [1,+\infty)$ be a Borel measurable function. For
a signed measure $\kap$, its {\it $\phi$-variation} is defined by
$\|\varkappa\|_{\phi,var}=\int_\XX\phi\,d|\varkappa|,$ where
$|\kap|=\kap_+=\kap_-$ is the variation of the signed measure
$\kap$. If $\phi\equiv 1$, the $\phi$-variation is the usual
total variation $\|\cdot\|_{var}$. Let $r:\ax\to \ax$ be some
function such that $r(t)\to 0, t\to \infty$.

\begin{dfn}\label{d11} The process $X$ is called  $(r,\phi)$-ergodic if the class of
invariant measures for $X$ contains exactly one measure $\pi$, and
$$
\|\mu_t-\pi\|_{\phi,var}\leq r(t)\int_{\XX}\phi\,d\mu,\quad
t\in\ax, \mu\in\Pf(\XX).
$$
\end{dfn}

We call the process  {\it exponentially $\phi$-ergodic} if there
exists some positive constants $C,\beta$ such that $X$ is
$(r,\phi)$-ergodic with $r(t)=Ce^{-\beta t}$.

 By
the common terminology, a {\it coupling} for a pair of the
processes $U,V$ is any two-component process $Z=(Z^1,Z^2)$ such
that $Z^1$ has the same distribution with $U$ and $Z^2$ has the
same distribution with $V$. Following this terminology, for every
$\mu,\nu\in\Pf$, we consider two versions $X^\mu,X^\nu$ of the
process $X$ with the initial distributions equal $\mu$ and $\nu$,
respectively, and  call a {\it $(\mu,\nu)$-coupling for the
process $X$} any two-component process $Z=(Z^1,Z^2)$ which is a
coupling for $X^\mu, X^\nu$.

\begin{dfn}\label{d13} The process $X$ \emph{admits an
exponential $\phi$-coupling} if there exists an invariant measure
$\pi$ for this process and constants $C_\phi>0,\beta>0$ such that,
for every $x\in \XX$, there exists a $(\delta_x,\pi)$-coupling
$Z=(Z^1,Z^2)$ with
$$
E\Big[\phi(Z^1_t)+\phi(Z_t^2)\Big]\1_{Z_t^1\not=Z_t^2}\leq C_\phi
e^{-\beta t}\phi(x),\quad t\geq 0.
$$
\end{dfn}

It is a simple observation that a process $X$ which admits an
exponential $\phi$-coupling  is exponentially $\phi$-ergodic.
This observation, however, gives an efficient tool for proving
exponential $\phi$-ergodicity, because explicit  sufficient
conditions are available that allow one to construct an
exponential $\phi$-coupling. Let us formulate one statement of
such a kind.

\begin{dfn}\label{d12} The process $X$ satisfies \emph{the local Doeblin
condition}, if  for
 every compact set $K\subset \XX$ there exists $T>0$ such that
 $$
 \kap(T,K)\eqdef \sup_{x,y\in K}{1\over 2}\|P_T(x,\cdot)-P_T(y,\cdot)\|_{var}<1.
 $$
\end{dfn}

\begin{prop}\label{p21}  Assume process $X$ to satisfy the local Doeblin condition. Let
function  $\phi:\XX\to [1,+\infty)$ be such that $\phi(x)\to
+\infty, x\to \infty$ and the process
\be\label{lyap0}\phi(X_t)+\int_0^t[\alpha\phi(X_s)-C]\, ds,\quad
t\in \ax\ee is a supermartingale w.r.t. to every measure $P_x, x\in
\XX$  for some positive constants $\alpha, C$.

Then the process $X$ admits an exponential $\phi$-coupling.
\end{prop}

Clearly, under conditions of Proposition \ref{p21}, the process
$X$ is exponentially $\phi$-ergodic.  The statements of such a
type are well known in the ergodic theory of Markov processes (see
e.g. \cite{And91} or \cite{MT93}), but usually the notion of a
$\phi$-coupling is not introduced separately. In section 3 below,
we show that  this notion is of independent interest because it
allows one to control  $L_p$ convergence rates as well.

\begin{rem}\label{r21} Frequently, the (exponential) ergodicity results
are formulated in the terms of other conditions that guarantee
irreducibility of the Markov process $X$ instead of  the local
Doeblin condition. For instance, in \cite{MT93},\cite{DFG09} such
an irreducibility condition is given in the terms of {\it petite
sets}. However, the local Doeblin condition appears to be more
convenient to deal with in the explicit construction of a
$\phi$-coupling. In addition, this condition can be verified
efficiently for important particular classes of processes, such
as diffusions (see \cite{Ver87}, \cite{Ver99} and section 7
below) or solutions to SDE's with jump noise (see \cite{Kul09}
and section 5 below).
\end{rem}

The more compact, but more restrictive, form of the above
condition on the function $\phi$ can be given in the terms of the
{\it extended generator} $\Af$ of the process $X$. Recall that a
locally bounded function $f:\XX\to \Re$ belongs to the domain
$Dom (\Af)$ of the extended generator $\Af$, if there exists a
locally bounded function $g:\XX\to \Re$ such that the process
$X^f_t\eqdef f(X_t)-\int_0^tg(X_s)\, ds, t\in \ax$ is a
martingale w.r.t. to any measure $P_x, x\in \XX$. For such $f$,
$\Af f\eqdef g$. Clearly,  process (\ref{lyap0}) is a
supermartingale w.r.t. to any measure $P_x, x\in \XX$ if the
function $\phi \in Dom (\Af)$ satisfies the following
Lyapunov-type condition: \be\label{lyap}
 \Af\phi(x)\leq -\alpha \phi(x)+C, \quad x\in \XX.
 \ee

Conditions of Proposition \ref{p21} appear to
be too restrictive for our further purposes; see more detailed discussion after
Proposition \ref{p22} below.  Thereby, we provide milder
conditions which still are sufficient  for $X$ to admit an exponential $\phi$-coupling.

For a closed set $K$ denote $\tau_K\eqdef\inf \{t\geq 0: X_t\in
K\}$, the hitting time of the set $K$ by the process $X$.

\begin{thm}\label{tA1} Assume process $X$ to satisfy the local Doeblin condition.
Let there exist function  $\phi:\XX\to [1,+\infty),$ compact set
$K\subset \XX$, and $\alpha>0$  such that
\begin{itemize}\item[1)] $\phi(x)\to +\infty, x\to \infty$;
\item[2)] $E_x\phi(X_t)\1_{\tau_K>t}\leq e^{-\alpha t}\phi(x),x\in\XX;$
\item[3)] $\sup_{x\in K, t\in\ax} E_x\phi(X_t)\1_{\phi(X_t)>c}\to 0,\quad c\to +\infty.$
\end{itemize}

 Then the process $X$ admits an
exponential $\phi$-coupling.
\end{thm}

Condition 1) of Theorem \ref{tA1} is quite natural as long as
$\phi$ is considered as a Lyapunov function. However,  in some
cases this condition  may be too restrictive, too. Below, we give
a version of Theorem \ref{tA1} that does not require any
assumptions on the limit behavior of $\phi$.

We  say that process $X$ {\it satisfies the Doeblin condition on a
set} $A\subset \XX$ if  there exists $T>0$ such that $\kap(T,
A)<1.$ We also say that process $X$ {\it satisfies the extended
Doeblin condition on a set} $A\subset \XX$ if  there exist $T_1,
T_2$ ($0<T_1<T_2$) such that \be\label{81}
 \kap(T_1,T_2,K)\eqdef \sup_{x,y\in K, s,t\in[T_1,T_2]}{1\over 2}\|P_s(x,\cdot)-P_t(y,\cdot)\|_{var}<1.
\ee
 We remark that these definitions are  not  standard ones,
but they look quite natural in the context of Definition \ref{d12}
and the following theorem.

\begin{thm}\label{tA2} Let there exist function  $\phi:\XX\to [1,+\infty),$ closed
set  $K\subset \XX$, and $\alpha>0$ such that conditions 2), 3) of
Theorem \ref{tA1} hold true. Assume that either
 $X$ satisfies the
Doeblin condition on  $\{\phi\leq c\}$  for every $c\in
[1,+\infty)$, or $X$  satisfies the extended Doeblin condition on
$K$.

 Then the process $X$ admits an
exponential $\phi$-coupling.
\end{thm}

From Theorems \ref{tA1}, \ref{tA2} we deduce the following
statement. Denote, for $t>0$, $$\tau_K^t\eqdef\inf \{s\geq 0:
X_{t+s}\in K\}.$$

\begin{prop}\label{p22}  Assume process $X$ to satisfy the local Doeblin
condition. Let there exist  compact set $K\subset \XX$ and
$S,\alpha>0$ such that
\begin{itemize}\item[1)] $\lim_{c\to+\infty}\lim\inf_{x\to \infty}P_x(\tau_K>c)>0$;
\item[2)] $E_xe^{\alpha\tau_K}<+\infty,x\in\XX;$
\item[3)] $\sup_{x\in K, t\in [0,S]} E_xe^{\alpha\tau_K^t}<+\infty$.
\end{itemize}

 Then, for every $\alpha'\in (0,\alpha),$ the process $X$ admits an
exponential $\phi$-coupling with $\phi(x)=E_xe^{\alpha'\tau_K},
x\in \XX$.

In addition, if $X$  satisfies the extended Doeblin condition on
$K$ then condition 1) is not required.
\end{prop}

This proposition demonstrates the difference between Theorems
\ref{tA1}, \ref{tA2} on one hand, and Proposition \ref{p21} on
another. Typically, a function $\phi$ of the type
$\phi(x)=E_xe^{\alpha\tau_K}$ neither belong to the domain of
$\Af$ nor satisfy condition of Proposition \ref{p21}. On the
other hand, Theorems \ref{tA1}, \ref{tA2} appear to be powerful
enough to handle the functions of such a type. This is important for our approach, since we would like to control the construction of a $\phi$-coupling  in the terms of the hitting times for the process $X$.

We prove  Theorems \ref{tA1}, \ref{tA2} and Proposition \ref{p22} in the Appendix.

\subsection{Semigroups generated by $X$: growth bounds and spectral properties of generators}

For a function $f:\XX\to \CC$, we denote
$$
T_t f(x)=\int_\XX f(y)P_t(x,dy), \quad t\in \XX, x\in\XX
$$
assuming the integrals to exist. Typically, the mapping $f\mapsto
T_tf$ forms a bounded linear operator in an appropriate
functional space.  We are mainly interested in the  functional
spaces $L_p\eqdef L_p^{\CC}(\XX,\pi), p\in(1,+\infty)$, but we
also consider some other auxiliary spaces. The
Chapman-Kolmogorov equation for the transition function
$P_t(x,dy)$ yields the semigroup property for the family
$\{T_t\}$: $T_{t+s}=T_tT_s, t,s\in \ax$. We assume process $X$ to
be stochastically continuous, which yields that $\{T_t\}$,
considered as a semigroup in $L_p$ with any $p\in(1,+\infty)$, is
strongly continuous. We denote by $A$ the generator of the
semigroup $\{T_t\}$. By the definition,
$$
Af\eqdef \lim_{t\to 0+}{1\over t}[T_tf-f],
$$
where the convergence holds in the sense of respective functional
space, and the domain of $A$ consists of all functions $f$ such
that the limit exists.

\begin{dfn}\label{d31} Let $\{T_t\}$ be a strongly continuous
semigroup of bounded linear operators on some complex Banach
space $\Xf$. A number $\gamma\in \Re$ is called
\begin{itemize}\item[a)] a \emph{spectral bound} for the gererator $A$ of
$\{T_t\}$, if every point $\lambda\in\CC$ with
$\mathrm{Re}\,\lambda>-\gamma$ belongs to the resolvent set of
$A$;

\item[b)] an (exponential) \emph{growth bound} for $\{T_t\}$, if there
exists $C\in\ax$ such that \be\label{32}\|T_tf\|_\Xf\leq C
e^{-\gamma t}\|f\|_\Xf,\quad  t\in \ax, f\in \Xf.\ee
\end{itemize}
\end{dfn}

The terminology introduced in Definition \ref{d31} differs
slightly from the standard one in the general spectral theory of
semigroups. Namely, the constant in the standard definition of a
growth bound may depend on $f$ (\cite{Nag86}, Chapter A-III).
However, this modified terminology appears to be more convenient
in our framework. We remark that the following condition is
equivalent to (\ref{32}) and, in some cases,  can be verified
more easily: \be\label{31}|\langle T_tf,g\rangle|\leq C
e^{-\gamma t}\|f\|_\Xf\|g\|_{\Xf^*},\quad t\in\ax, f\in \Xf, g\in
\Xf^*,\ee where $\Xf^*$ is the dual space for $\Xf$.

The following statement is quite standard, but, for the sake of completeness, we give the sketch of the proof here.

\begin{prop}\label{p31} If $\gamma\in\Re$ is a growth bound for $\{T_t\}$, then $\gamma$ is a
spectral bound for its generator.
\end{prop}

\begin{proof} Take any $\lambda\in\CC$ with
$\mathrm{Re}\,\lambda>-\gamma$ and consider the mapping
$$
R_\lambda:f\mapsto \int_0^\infty e^{-\lambda s}T_sf\,
ds\eqdef\lim_{S\to +\infty} \int_0^S e^{-\lambda s}T_sf\, ds.
$$
The integrals under the limit are defined in the Riemannian sense,
and the limit  exists in the sense of the norm convergence. The
operator $R_\lambda$ is bounded with its norm being dominated by
$C( \mathrm{Re}\, \lambda+\gamma)^{-1}$, where the constant $C$
comes from the definition of a growth bound. On the other hand,
by standard arguments, $T_tR_\lambda=R_\lambda T_t=e^{\lambda
t}\left(R_\lambda-\int_0^te^{-\lambda s}T_s\, ds\right)$ and
$R_\lambda$ is the inverse operator for $\lambda-A$, i.e.,
$\lambda$ belongs to the resolvent set of $A$.
\end{proof}

 For semigroups defined by a (conservative) Markov process $X$
in $L_p, p\in(1,+\infty)$, the point $\lambda=0$ is a trivial
eigenvalue with the corresponding eigenfunction $f_\lambda=\1$
(i.e., the function that equals 1 in every point). If this
eigenvalue is simple  and the rest of the spectrum of the
generator $A$ is separated from zero, then it is said that this
generator (resp., semigroup or process) possesses a {\it spectral
gap}. This motivates the following terminology. Denote, for
$p\in(1,+\infty)$,
$$
L_p^0=\{f\in
L_p^\CC(\XX,\pi):\int_Xfd\pi=0\}=\langle\1\rangle^\perp,
$$
where in the last expression $\1$ is interpreted as an element of
$L_p^*=L_q, q^{-1}+p^{-1}=1$. Since $\pi$ is an invariant measure
for $X$, one has
$$
\int_\XX T_tf(x)\pi(dx)=\int_\XX\int_\XX
f(y)P_t(x,dy)\pi(dx)=\int_\XX f(y)\, \pi(dy),
$$
which means, in particular, that $L_p^0$ is invariant under
$\{T_t\}$.

For a given $\gamma>0, p\in(1,+\infty)$, we say that process $X$
possesses either property
 $SG_p(\gamma)$ or property $GB_p(\gamma),$  if, for the restriction
 of its semigroup $\{T_t\}$ to the space $\Xf=L_p^0$, the
number $\gamma$ is a spectral bound or a  growth bound,
respectively. Also, we say that the process
possesses an exponential $L_p$ rate if
\be\label{330}\|T_tf\|_{p}\leq C e^{-\gamma t}\|f\|_{p},\quad t\in
\ax, f\in L_2^0\ee with some $C>0, \gamma >0$ (here and below, we denote $\|\cdot\|_p\eqdef
\|\cdot\|_{L_p}$).

Some authors (e.g. \cite{Chen00}) say that the process $X$
possesses an exponential $L_2$ rate if
\be\label{33}\|T_tf\|_{2}\leq e^{-\gamma t}\|f\|_{2},\quad t\in
\ax, f\in L_2^0\ee with some $\gamma>0$. This terminology does
not look to be perfectly adjusted with the matter of the problem
discussed above, because the constant $C$ in (\ref{330}) with
$p=2$ does not play an essential role in the asymptotic behavior
of the semigroup; in particular, the estimate (\ref{330}) is
already strong enough to provide existence of a spectral gap for
$X$.

On the other hand, it makes sense to consider estimate of the type (\ref{33}) separately. Let us express (\ref{33}) in the terms of the {\it Dirichlet
form} $\Ef$ associated with the process $X$. Recall that the Dirichlet form
$\Ef$ corresponding to the $L_2$-semigroup $\{T_t\}$ generated by
$X$ is  defined as  the completion of the  bilinear form
$$
Dom(A)\times Dom(A)\ni (f,g)\mapsto -(Af,g)_{L_2}
$$
with respect to the norm $\|\cdot\|_{\Ef,1}\eqdef
\Big[\|\cdot\|_{L_2}^2-(A\cdot,\cdot)_{L_2}\Big]^{1\over 2}$ (e.g.
\cite{MR92}, Chapter 2). It can be verified easily that, for
$c=\gamma^{-1}$, (\ref{33}) is equivalent to the functional
inequality \be\label{34} \int_\XX |f|^2d\pi-\left|\int_\XX f
d\pi\right|^2\leq c \,\Ef(f,f),\quad f\in Dom(\Ef),\ee called the
{\it Poincar\'e inequality}. The Poincar\'e inequality is one of
the most important in the field, and this motivates the interest
to the inequality (\ref{33}). One can say that (\ref{33}) is a
kind of a differential estimate, while (\ref{330}) with $p=2$ is
an integral one. In section 4 below we give an example which
demonstrates that these estimates are non-equivalent.

For a given $\gamma>0$, we say that
process $X$ possesses the property $PI(\gamma)$ if (\ref{33})
holds true. We have the following implications:
$$
PI(\gamma)\Rightarrow GB_2(\gamma),\quad  GB_p(\gamma) \Rightarrow
SG_p(\gamma).
$$
Examples are available, where a number being a spectral bound is
not a growth bound (\cite{Nag86}, Example 1.4, \cite{Chen00},
Example 2.3), and thus $SG_p(\gamma)\not\Rightarrow
GB_p(\gamma)$. As we have already mentioned,
$GB_2(\gamma)\not\Rightarrow PI(\gamma)$ (see section 4).
 Therefore, in general, each of three
properties formulated above requires a separate investigation.

\section{$L_p$ convergence rates and Poincar\'e inequality for a  process that admits an exponential $\phi$-coupling}

In this section, we assume that, for a given function $\phi$, the
process $X$ admits an exponential $\phi$-coupling.  We denote by
$C_\phi,\beta$ the constants from the definition of a
$\phi$-exponential coupling, and write $C$ for any constant which
can be, but is not, expressed explicitly. The value of the
constant $C$ can vary from line to line. We denote by $\pi$ the
unique invariant measure for $X$ and assume $\phi\in
L_1(\XX,\pi)$. This assumption is not restrictive;  it holds true
under conditions of either Theorem \ref{tA1}, Theorem \ref{tA2},
or Proposition \ref{p22} (see Remark \ref{rA1} in the Appendix).

We separate  our investigation into several parts. First, we
establish rates of convergence of the semigroups generated by $X$
in auxiliary spaces $L_{p,\phi}, L_{p,\phi}^*$. Then we consider
$L_p$-semigroups with arbitrary $p\in (1+\infty)$. Finally, we
investigate the $L_2$-semigroup, considering separately the cases
of a reversible and an irreversible (in time) process $X$ separately.

\subsection{Spaces $L_{p,\phi}, L_{p,\phi}^*, p\in(1,+\infty)$.}

For $p\in(1,+\infty)$, denote by $L_{p,\phi}, p\in(1,+\infty)$
the set of functions $f$ such that
$$
\|f\|_{p,\phi}\eqdef\left[\int_\XX \left|{f\over \phi^{1\over
q}}\right |^p \,d\pi\right]^{1\over p}<+\infty,
$$
where $q$ is adjoint to $p$, i.e. $p^{-1}+q^{-1}=1$. The set
$L_{p,\phi}$ is a Banach space with the norm
$\|\cdot\|_{p,\phi}$. The dual space $L_{p,\phi}^*$ to
$L_{p,\phi}$ with respect to the natural duality $(f,g)\mapsto
\langle f,g\rangle\eqdef \int_\XX f\bar gd\pi$ coincides  with
the space of functions $f$ such that
$$
\|f\|_{p,\phi}^*\eqdef\left[\int_\XX |f|^q
\phi\,d\pi\right]^{1\over q}<+\infty.
$$

The space   $L_{p,\phi}^*$ is a subset of $L_q$ since $\phi\geq
1$. On the other hand, $\phi$ may be unbounded, and in this case
$L_{p,\phi}$ is strictly larger than $L_p$.  Nevertheless, in any
case $L_{p,\phi}\subset L_1$ because
$$
\int_\XX|f|d\pi\leq \left[\int_\XX\left|{f\over \phi^{1\over
q}}\right|^p\,d\pi\right]^{1\over p} \left[\int_\XX |\phi^{1\over
q}|^q\,d\pi\right]^{1\over q}=\|f\|_{p,\phi}\|\phi\|_{L_1}^{1\over
q}.
$$
Define $L_{p,\phi}^0$ and $L_{p,\phi}^{*,0}$ as the subspaces of
the elements $f$ of $L_{p,\phi}$ and $L_{p,\phi}^{*}$, respectively,
such that  $\int_Xf\,d\pi=0$. Clearly, $L_{p,\phi}^{*,0}$ is the
dual space to $L_{p,\phi}^0$ w.r.t. duality $\langle
\cdot,\cdot\rangle$.

\begin{thm}\label{t41} For every $p\in (1,+\infty)$, $\{T_t\}$ is a semigroup of bounded
operators in $L_{p,\phi}$. The subspace $L_{p,\phi}^0$ is
invariant w.r.t. to $\{T_t\}$, and ${\beta\over
q}=\beta-{\beta\over p}$ is a growth bound for the restriction of
$\{T_t\}$ on $L_{p,\phi}^0$. \end{thm}

\begin{proof} In the representation $L_{p,\phi}=\langle \1\rangle\bigoplus L_{p,\phi}^0$, both summands
are invariant subspaces for the semigroup $\{T_t\}$.  Clearly,
every  $T_t$ is an identity operator on the one-dimensional
subspace $\langle \1\rangle$. Let us investigate the restriction
of $\{T_t\}$ on $L_{p,\phi}^0$.

Let us prove that, for every $f\in L_{p,\phi}^0$ and $x\in \XX$,
\be\label{41} |T_tf(x)|^p\leq 2^{p-1}\Cs_\phi^{p\over q}
e^{-{\beta p\over q}t}\phi^{p\over q}(x)\left(T_t\left({|f|^p\over
\phi^{p\over q}}\right)(x)+ \|f\|^p_{p,\phi}\right). \ee Consider
an exponential $\phi$-coupling $Z=(Z^1,Z^2)$ that exists by
assumption. We have
$$
T_tf(x)=T_tf(x)-\int_\XX f(y)\pi(dy)=E\Big[f(Z^1_t)-f(Z^2_t)\Big],
$$
here in the last equality we have used that $\pi$ is an invariant
measure and thus $Z_t^2$ has the distribution $\pi$ for every
$t$. Then
$$\begin{aligned}
|T_tf(x)|^p&=\left|E\Big[f(Z^1_t)-f(Z^2_t)\Big]\1_{Z_t^1\not=Z_t^2}\right|^p\leq
E{|f(Z^1_t)-f(Z^2_t)|^p\over [\phi(Z^1_t)+\phi(Z_t^2)]^{p\over
q}}\\
&\times\left(E\Big[\phi(Z^1_t)+\phi(Z_t^2)\Big]\1_{Z_t^1\not=Z_t^2}\right)^{p\over
q}\leq
 C_\phi^{p\over q} e^{-{\beta p\over q}t}\phi^{p\over q}(x)E{|f(Z^1_t)-f(Z^2_t)|^p\over [\phi(Z^1_t)+\phi(Z_t^2)]^{p\over
q}}.
\end{aligned}
$$
We have $\phi(Z^1_t)+\phi(Z_t^2)\geq \max
\Big(\phi(Z^1_t),\phi(Z_t^2)\Big)$. Hence,
$$
E{|f(Z^1_t)-f(Z^2_t)|^p\over [\phi(Z^1_t)+\phi(Z_t^2)]^{p\over
q}}\leq 2^{p-1}E\left[{|f(Z^1_t)|^p \over \phi^{p\over
q}(Z^1_t)}+{|f(Z^2_t)|^p \over \phi^{p\over
q}(Z^2_t)}\right]=2^{p-1} T_t\left({|f|^p\over \phi^{p\over
q}}\right)(x)+ 2^{p-1}\|f\|^p_{p,\phi},
$$
which proves (\ref{41}).

As a corollary, we get the following  estimate valid for every
$f\in L_{p,\phi}^0$: \be\label{42}\ba \|f\|_{p,\phi}^p&=\int_\XX
|T_t f(x)|^p\phi^{-{p\over q}}(x)\pi(dx)\leq
2^{p-1}C_\phi^{p\over q} e^{-{\beta p\over q}t}\int_\XX
\left(T_t\left({|f|^p\over \phi^{p\over q}}\right)(x)+
\|f\|^p_{p,\phi}\right)\pi(dx)\\
& = 2^{p-1}C_\phi^{p\over q} e^{-{\beta p\over q}t}\left(\int_\XX
{|f(x)|^p\over \phi^{p\over
q}(x)}\pi(dx)+\|f\|^p_{p,\phi}\right)=2^{p}C_\phi^{p\over q}
e^{-{\beta p\over q}t}\|f\|^p_{p,\phi}\ea \ee (here, the
invariance property for $\pi$ is used).

By (\ref{42}), every $T_t$ is bounded on $L_{p,\phi}^0$, and thus
it is bounded on whole $L_{p,\phi}.$ Moreover, (\ref{42})
immediately implies  inequality (\ref{32}) from the definition of
a growth bound.
\end{proof}

By standard duality arguments, Theorem \ref{t41} yields the
following corollary for the adjoint semigroup $\{T_t^*\}$.

\begin{cor}\label{c41}
For every $p\in (1,+\infty)$, $\{T_t^*\}$ is a semigroup of
bounded operators in $L_{p,\phi}^*$. The subspace
$L_{p,\phi}^{*,0}$ is invariant w.r.t. to $\{T_t\}$, and
${\beta\over q}=\beta-{\beta\over p}$ is a  growth bound for the
restriction of $\{T_t\}$ on $L_{p,\phi}^{*,0}$.
\end{cor}

\subsection{Spaces $L_{p},  p\in(1,+\infty)$.}

Theorem \ref{t41}, in fact, provides that the generator of
$\{T_t\}$, considered as a semigroup in $L_{p,\phi}$, possesses a
spectral gap. The following simple corollary shows that, in a particular case,  this yields
existence of a spectral gap for the generator of the respective  $L_p$-semigroup.

\begin{cor}\label{c43} If the function $\phi$ is bounded, then the process $X$ satisfies $GB_p\left(\beta-{\beta\over p}\right), p\in (1,+\infty)$.
\end{cor}

\begin{proof} Since $1\leq \phi\leq C$, the norms $\|\cdot\|_{p,\phi}$ and $\|\cdot\|_p$ are equivalent.\end{proof}

However, the  general situation is more complicated, and under conditions of Theorem \ref{t41} respective $L_p$ generators may fail to possess
 a spectral gap (see
section 4). Here we provide existence of a spectral gap  under
additional assumptions formulated in the terms of the {\it dual
process} $X^*$ to the Markov process $X$.

Recall that if $\pi$ is  an invariant measure for the Markov
process $X$, then, on appropriate probability space, a stationary
process $\tilde X_t, t\in \Re$ can be constructed in such a way
that $\tilde X_0\sim \pi$ and $\tilde X$ is a Markov process with
the transition function $P_t(x,dy)$. The process $X^*_t\eqdef
\tilde X_{-t}, t\in\Re$ is again a time-homogeneous Markov
process. The process $X^*$ is called the dual process for $X$. By
stationarity,
$$
\langle T_tf,g\rangle=Ef(\tilde X_t)\overline{g(\tilde X_0)}=E
f(\tilde X_0)\overline{g(\tilde
X_{-t})}=Ef(X^*_0)\overline{g(X_t^*)}, \quad f\in L_p, g\in L_q,
$$ hence the adjoint semigroup $\{T_t^*\}$ for the semigroup
$\{T_t\}$ generated by $X$ in $L_p$ coincides with the semigroup
generated by $X^*$ in $L_q$.

For a functions $\phi,\psi:\XX\to [1,+\infty)$,  we write
$\phi\asymp\psi$ if
$$
\inf_x{\phi(x)\over \psi(x)}>0,\quad  \sup_x{\phi(x)\over
\psi(x)}<+\infty.
$$

\begin{thm}\label{t43} Assume that there exist functions $\phi$ and $\phi^*$ such
that the process $X$ admits an exponential $\phi$-coupling, the
dual process $X^*$ admits an exponential $\phi^*$-coupling, and
$\phi\asymp\phi^*$.

Then, for every $p\in[2,+\infty)$ and $\gamma<{\beta \over
2p-1}$, process  $X$ satisfies $GB_p\left(\gamma\right)$.
\end{thm}

\begin{proof} We will show that \be\label{44} |\langle
T_tf,g\rangle|\leq C e^{-\gamma t}\|f\|_p\|g\|_q, \quad f\in
L_p^0,g\in L_q^0, t\in[1,+\infty).\ee Since $T_t$ is a
contraction semigroup in $L_p$, this will provide that (\ref{31})
holds true for the restriction of $\{T_t\}$ to $L_p^0$ (with some
other constant $C$), and thus will prove the required statement.

Let us verify first that \be\label{45} \left|\langle
T_tf,g\rangle-\int_\XX f\,d\pi\int_\XX \bar g\, d\pi\right|\leq C
e^{-{\beta\over p} t}\|f\|_{p,\phi}\|g\|_{p,\phi}^*,\quad  f\in
L_{p,\phi}, g\in L_{p,\phi}^{*}. \ee For $f\in L_{p,\phi}^0, g\in
L_{p,\phi}^{*,0}$, inequality (\ref{45}) with $C=2C_\phi^{1\over
q}$ follows from (\ref{42}). For arbitrary $f\in L_{p,\phi}, g\in
L_{p,\phi}^{*}$, one has
$$
\langle T_tf,g\rangle-\int_\XX f\,d\pi\int_\XX \bar g\,
d\pi=\langle \Pi f, \Pi g\rangle,
$$
where $\Pi f\eqdef f-\int_\XX f\, d\pi$. Since $\Pi$ is bounded
both as an operator $L_{p,\phi}\to L_{p,\phi}^0$ and as an
operator $L_{p,\phi}^*\to L_{p,\phi}^{*,0}$, this yields
(\ref{45}) in the general case.

Let us proceed with the proof of (\ref{44}). Take $\gamma\in
\left(0,{\beta\over 2p-1}\right).$ Denote $I_k=\{x: e^{\gamma
kt}\leq \phi(x)< e^{\gamma (k+1)t}\},f_k=f I_k, g_k=g I_k, k\geq
0$. We have \be\label{46} \ba\langle
T_tf,g\rangle&=\sum_{k,j=0}^\infty\langle T_tf_k,g_j\rangle=
\sum_{k=0}^\infty \langle T_tf_k,g_k\rangle\\
&+\sum_{r=1}^\infty\sum_{k=0}^\infty\langle
T_tf_{k+r},g_{k}\rangle+\sum_{r=1}^\infty\sum_{k=0}^\infty\langle
T_tf_k,g_{k+r}\rangle.\ea \ee Let us estimate the summands in the
right hand side of (\ref{46}) separately.

We  have from (\ref{45})
$$
|\langle T_tf_k,g_k\rangle|\leq \left|\int_\XX
f_k\,d\pi\right|\left|\int_\XX g_k\, d\pi\right|+C
e^{-{\beta\over p} t}\|f_k\|_{p,\phi}\|g_k\|_{p,\phi}^{*}.
$$
By the construction, $g_k=0$ on the set
$\{\phi>e^{\gamma(k+1)t}\}$, thus
$$
\|g_k\|_{p,\phi}^*=\left[\int_\XX |g_k|^q\phi\,
d\pi\right]^{1\over q}\leq \|g_k\|_{q}\cdot e^{{\gamma(k+1)\over
q}t}.$$ Analogously, $$\|f_k\|_{p,\phi}=\left[\int_\XX
|f_k|^p\phi^{-{p\over q}}\, d\pi\right]^{1\over p}\leq
\|f_k\|_{p}\cdot e^{-{\gamma k\over q}t}.$$ For $k\geq 1$, we have
$$
\left|\int_\XX f_k\, d\pi\right|=\left|\int_{I_k} f\,
d\pi\right|\leq \|f\|_p\pi^{1\over q}(I_k)\leq
\|f\|_p\|\phi|_1^{1\over q}e^{-{\gamma k\over q}t}.
$$
For $k=0$,  since $\int_\XX f\, d\pi=0$,
$$
\left|\int_\XX f_k\, d\pi\right|=\left|\int_{\XX\backslash I_0}
f\, d\pi\right|\leq \|f\|_p\pi^{1\over q}(\XX\backslash I_0)\leq
\|f\|_p\|\phi|_1^{1\over q}e^{-{\gamma \over q}t}.
$$
Analogously,
$$
\left|\int_\XX g_k\, d\pi\right|\leq \|g_k\|_q\|\phi|_1^{1\over
p}\min[e^{-{\gamma k\over p}t},e^{-{\gamma \over p}t}].
$$
Therefore,
$$
\sum_{k=0}^\infty |\langle T_tf_k,g_k\rangle|\leq C(e^{-\gamma
t}+e^{-{\beta\over p}  t+{\gamma \over
q}t})\sum_{k=0}^\infty\|f_k\|_p\|g_k\|_q\leq
$$
$$
\leq C(e^{-\gamma t}+e^{-{\beta\over p} t+{\gamma \over q}t})
\left[\sum_{k=0}^\infty\|f_k\|_p^p\right]^{1\over
p}\left[\sum_{k=0}^\infty\|g_k\|_q^q\right]^{1\over
q}=C(e^{-\gamma t}+e^{-{\beta\over p} t+{\gamma \over
q}t})\|f\|_p\|g\|_q.
$$
In the last equality, we have used that the family $\{I_k\}$ is
disjoint, and hence
$$
\sum_k\|f_k\|_p^p=\sum_k\int_{I_k}|f|^p\, d\pi=\|f\|_p^p,\quad
\sum_k\|g_k\|_q^q=\sum_k\int_{I_k}|g|^q\, d\pi=\|g\|_q^q.
$$
By the choice of $\gamma$, we have ${\beta\over p}-{\gamma\over
q}>\gamma\left({2p-1\over p}-{1\over
q}\right)=\gamma\left(2-p^{-1}-q^{-1}\right)=\gamma$. Therefore,
finally, \be\label{47}\sum_{k=0}^\infty |\langle
T_tf_k,g_k\rangle|\leq Ce^{-\gamma t}\|f\|_p\|g\|_q, \quad f\in
L_p^0,g\in L_q^0, t\in\ax. \ee

Analogously, for every $k\geq 0,r\geq 1$, we have
$$
\ba |\langle T_tf_{k+r},g_k\rangle|&\leq C\|f_{k+r}\|_p
\|g_k\|_q\left(e^{-{\gamma(k+r)\over q}t}\min(e^{-{\gamma\over p}
t},e^{-{\gamma k\over p} t})+e^{-{\beta\over p}t+{\gamma(k+1)\over
q}t-{\gamma(k+r)\over q}t} \right)\\
 &\leq C\|f_{k+r}\|_p \|g_k\|_q e^{-\gamma t}e^{-{\gamma(r-1)\over
q}t}.\ea
$$
For every given $r\geq 1$,
$$
\sum_{k=0}^\infty\|f_{k+r}\|_p\|g_k\|_q\leq\left[\sum_{k=0}^\infty\|f_{k+r}\|_p^p\right]^{1\over
p}\left[\sum_{k=0}^\infty\|g_k\|_q^q\right]^{1\over q}\leq
\|f\|_p\|g\|_q.
$$
In addition,
$$
\sum_{r=1}^\infty e^{-{\gamma(r-1)\over q}t}\leq
[1-e^{-{\gamma\over q}}],\quad t\in[1,+\infty).
$$
Hence, finally,
\be\label{48}\sum_{r=1}^\infty\sum_{k=0}^\infty|\langle
T_tf_{k+r},g_{k}\rangle|\leq Ce^{-\gamma t}\|f\|_p\|g\|_q, \quad
f\in L_p^0,g\in L_q^0, t\in[1,+\infty). \ee

Up to this moment, we have not used the assumption that $X^*$
admits an exponential $\phi^*$-coupling.  Now, we use this
assumption in order to estimate the last  summand in the right
hand side of (\ref{46}). We replace $X,\phi,p$ by $X^*,\phi^*,q$
and write,  under this assumption,  the following estimate
analogous to (\ref{45}):
$$
\left|\langle T_t^* g,f\rangle-\int_\XX g\,d\pi\int_\XX \bar f\,
d\pi\right|\leq C e^{-{\beta\over q}
t}\|g\|_{q,\phi^*}\|f\|_{q,\phi^*}^*,\quad  g\in L_{q,\phi^*},
f\in L_{q,\phi^*}^{*}.
$$
From the condition $\phi\asymp\phi^*$ we conclude that
$$
\|g_k\|_{q,\phi^*}=\left[\int_{I_k} |g_k|^q(\phi^*)^{-{q\over
p}}\, d\pi\right]^{1\over q}\leq C \|g_k\|_{q}\cdot e^{-{\gamma
k\over p}t},
$$
$$
\|f_k\|_{q,\phi^*}^*=\left[\int_{I_k} |f_k|^p\phi^*\,
d\pi\right]^{1\over p}\leq C \|f_k\|_{p}\cdot e^{{\gamma
(k+1)\over p}t}.
$$
In addition, by the choice of $\gamma$ we have ${\beta\over
q}-{\gamma\over p}>\gamma\left({2p-1\over q}-{1\over
p}\right)=\gamma\left(2{p\over q}-p^{-1}-q^{-1}\right)\geq
\gamma$ because $p\geq 2\geq q$. Then estimates analogous to
those made above yield
\be\label{49}\sum_{r=1}^\infty\sum_{k=0}^\infty|\langle
T_tf_{k},g_{k+r}\rangle|=\sum_{r=1}^\infty\sum_{k=0}^\infty|\langle
T_t^*g_{k+r}, f_k\rangle|\leq Ce^{-\gamma t}\|f\|_p\|g\|_q,  \ee
$f\in L_p^0,g\in L_q^0, t\in[1,+\infty).$ Now (\ref{44}) follows
from (\ref{46}) -- (\ref{49}). The theorem is proved.\end{proof}

\begin{cor}\label{c42} Under conditions of Theorem \ref{t43}, for
$p\in (1,+\infty)$, process $X$ satisfies $GB_p({\beta\over
2p-1}\wedge {\beta(p-1)\over p+1})$.
\end{cor}
\begin{proof} Conditions of Theorem \ref{t43} are symmetric w.r.t. the choice between the process $X$ and its dual process $X^*$.
It is clear that the property $GB_p(\gamma)$ for $X$ is
equivalent to the property $GB_q(\gamma)$ for $X^*$. Therefore,
for $p\in(1,2]$, the process  $X$ satisfies
$GB_p\left(\gamma\right)$ for every $\gamma<{\beta \over
2q-1}={p-1\over p+1}\beta.$
\end{proof}

\subsection{Space $L_{2}$: the Poincar\'e inequality.}

Theorem \ref{t43} does not give any information about the
property $PI$. In this section, we investigate this property
separately. Recall that we assume  the process $X$ to admit an
exponential $\phi$-coupling.

Assume first that the process $X$  is  time-reversible;  that is,
the  dual process $X^*$ has the same distribution with $X$. The
following theorem looks quite standard (see \cite{RR97}, Theorem
2.1 or \cite{Chen00}, Theorem 1.2 for similar statements).
Nevertheless, even in the most studied case of a diffusion
process $X$, this statement gives rise for a
 new criteria for the Poincar\'e inequality  (see section 7 below).

\begin{thm}\label{t42} Consider the semigroup $\{T_t\}$ generated by $X$ in $L_2$.
Assume that $X$ is time reversible, or, equivalently, $T_t=T_t^*,
t\in \ax$.

Then $X$ satisfies $PI\left({\beta\over 2}\right)$.
\end{thm}
\begin{proof} Since  $T_t, t\in\ax$ is a contraction semigroup of self-adjoint
non-negative operators, it can be represented as $T_t=e^{-At},
t\in\ax$, where $A$ is the $L_2$-generator of the process $X$,
and $A$ is self-adjoint and non-negative. Let $P(d\lambda)$ be the
projector-valued measure from  the spectral decomposition for the
operator $A$:
$$
A=\int_0^\infty\lambda P(d\lambda).
$$
Then
$$
T_t=\int_0^\infty e^{-\lambda t}P(d\lambda),
$$
and, for every $f\in L_2$, \be\label{43}
\|T_tf\|^2_2=\int_0^\infty e^{-\lambda t} (P(d\lambda)f,f),\quad
t\in\ax,\ee
 where $(\cdot,\cdot)$ denotes scalar product in $L_2$.

We have $\|\cdot\|_2\leq \|\cdot\|_{2,\phi}^*$, and hence
Corollary \ref{c41} provides that, for every $f\in
L_{2,\phi}^{*,0},$ there exists a constant $C(f)\in \ax$ such that
$$
\|T_tf\|^2_{L_2}\leq C(f)e^{-\beta t},\quad t\in\ax.
$$
The latter inequality and (\ref{43}) implies that, for such $f$,
the measure $(P(d\lambda)f,f)$ is supported by $[\beta,+\infty)$.
One has  $(P(\Delta)f_n,f_n)\to (P(\Delta)f,f)$ for every Borel set
$\Delta$ and every sequence $f_n\to f$ in $L_2$. In addition, the
set $L_{2,\phi}^{*,0}$ is dense in
$L_2^0\eqdef\langle\1\rangle^\perp.$ Therefore,  the measure
$(P(d\lambda)f,f)$ is supported by $[\beta,+\infty)$ for every
$f\in L_2^0$, and (\ref{43}) yields (\ref{33}) with
$\gamma={\beta\over 2}$.
\end{proof}

Note that all the properties $SG_2(\gamma), GB_2(\gamma)$,
and $PI(\gamma)$ coincide for a time-reversible process $X$.
 This can be verified easily, and the argument
here is similar to the previous proof. The spectral decomposition
theorem is the key tool here, and the claim  for the generator
$A$ of $\{T_t\}$ to be self-adjoint (or, at least, normal) is
crucial. This claim is closely related with the structure of the
process. For instance, it is satisfied when $X$ is a diffusion
process. On the other hand, for $X$ being a solution to SDE with
a  jump noise, this claim is highly restrictive. This motivates
the following modification of Theorem \ref{t42}, that extends the
domain of its applications. The construction exposed below is an
appropriate modification of the one introduced in  \cite{Chen00}.

The rough idea is to replace the $L_2$-generator $A$ by the
operator $A^\diamond\eqdef{1\over 2}(A+A^*)$. Since the symmetric
part of the Dirichlet form generated by $A$ coincides with the
Dirichlet form generated by $A^\diamond$, the $PI(\gamma)$
property for the process $X$ would be  equivalent to the
$PI(\gamma)$ property for the process $X^\diamond$ corresponding
to $A^\diamond$. In the formal realisation of this idea, one
needs, at least,  to take care of the domains of various
generators.

\begin{thm}\label{t44} Assume there exists a time-reversable Markov process $X^\diamond$
that admits a $\phi$-coupling for some $\phi$. Assume also that
there exists a set $\Df\subset L_2$ such that

(i) $\Df\cap L_2^0$ is dense in $L_2^0$;

(ii) $\Df$ is invariant w.r.t. $L_2$-semigroup corresponding to
$X$;

(iii) $\Df$ belongs to the domains to the $L_2$-generators
$A,A^*,$ and $A^\diamond$  corresponding to $X,X^*,$ and $
X^\diamond$, respectively, and
$$
A^\diamond f={1\over 2}(Af+A^*f),\quad f\in \Df.
$$

Then $X$ satisfies $PI_2\left({\beta\over 2}\right)$, where
$\beta$ is the constant from the definition of the $\phi$-coupling
for $X^\diamond$.
\end{thm}

\begin{proof} Denote by $\{T_t^\diamond\}$ the $L_2$-semigroup
generated by $X^\diamond$. It follows from the previous theorem
that, for every $f\in \Df\cap L_2^0$,
$$
(A^\diamond f,f)\leq -{\beta\over 2} \|f\|^2_2.
$$
Since $(A^\diamond f,f)={1\over
2}[(Af,f)+(A^*f,f)]=\mathrm{Re}\,(Af,f)$, this yields
$$
\mathrm{Re}\,(Af,f)\leq -{\beta\over 2} \|f\|^2_2, \quad f\in
\Df\cap L_2^0.
$$
Then, for $f\in \Df\cap L_2^0,$ we have
$$
{d\over dt}\|T_tf\|^2=2\mathrm{Re}\,(AT_tf, T_tf)\leq -\beta
\|T_tf\|^2_2, \quad t\in\ax,
$$
here we have used that, by the condition (ii), $T_tf\in \Df\cap
L_2^0$. Hence, \be\label{418} \|T_tf\|^2_2\leq e^{-\beta
t}\|f\|_2^2\ee for every $f\in \Df\cap L_2^0$. Since $\Df\cap
L_2^0$ is dense in $L_2^0$, (\ref{418}) holds true for every $f\in
L_2^0$.
\end{proof}

\section{One example}

In this section, we give an example of a Markov process which
demonstrates relations between the objects considered in our main
exposition.  We will see that process that admits an exponential
$\phi$-coupling may fail to possess a spectral gap property. This
would make more clear the statements of section 3.2: in general,
in order to control growth bounds and spectral properties of
$L_p$ semigroups, one should control ergodic properties  both for
the process $X$ and for the dual process $X^*$. Also, we will see
that the exponential $L_2$ growth bound (\ref{330}) {\it is not
equivalent} to the Poincar\'e inequality (\ref{33}).
Consequently, for time-irreversible  processes these two
inequalities should be studied separately.

Let $\XX=[0,+\infty)$ and the extended generator of the process $X$ be defined on the functions $f\in C^1$ by the formula
$$
\Af f(x)=-a(x)f'(x)+\theta(x)\sum_{k=1}^\infty (1-p)p^{k-1}[f(x_k)-f(x)],\quad x\in \XX,
$$
where  $\{x_k, k\geq 1\}\subset [1,+\infty)$, $p\in(0,1)$, and  $a,\theta\in C^1$ are functions taking values in $[0,1]$. We assume that
$$a(0)=0,\,\, a(x)>0, \,x>0, \,\, a(x)=1,\, x\geq 1\, \hbox{  and }\,\theta(x)=0,\, x\geq 1, \,\, \theta(x)=1,\,  x\leq {1\over 2}.
 $$It is also assumed  that $x_k< x_{k+1}, k\geq 1$; that is, the points $x_k,k\geq 1$ are naturally ordered.

The dynamics of the process $X$ contains two parts. The first
(deterministic) component is given by the ordinary differential
equation (ODE) $dx=-a(x)dt$. The second (jump) part corresponds
to possibility for the process to jump at one of the positions
$x_k, k\geq 1$. The intensity for such a jump depends on $k$ and
the current position $x$, and is equal $(1-p)p^{k-1}\theta(x)$.

For this model, ergodic and spectral properties can be expressed  completely
in the terms of $p$ and $\{x_k, k\geq 1\}$; let us formulate corresponding statements.

\begin{enumerate}
\item If there exists $\alpha>0$ such that $\sum_{k\geq 1}p^{k}e^{\alpha x_k}<+\infty,$
then $X$ admits an exponential $\phi$-coupling with $\phi(x)=e^{\alpha x}$.

\item Condition $\sup_k(x_{k+1}-x_k)<+\infty$ is necessary for $X$ to satisfy $SG_p(\gamma)$
with some $\gamma>0$, and sufficient for $X$ to satisfy $GB_p(\gamma')$ with some $\gamma'>0$.
\item For any sequence $\{x_k, k\geq 1\}$, the process $X$ does not satisfy the Poincar\'e inequality.
\end{enumerate}

\emph{Proof of statement (1).} It is clear that the Lyapunov-type condition (\ref{lyap}) holds true with $\phi(x)=e^{\alpha x}$. Hence, it is enough to prove that the local Doeblin condition holds and then use Proposition \ref{p21}. Denote by $\psi_{t}(x), t\in \ax, x\in \XX$  the flow generated by ODE $dx=-a(x)dt.$ For a given compact $K\subset \XX$,  there exists $T_K>0$ such that $\psi_t(x)\leq {1\over 2}, x\in K, t\geq T_K$. This together with the Chapman-Kolmogorov equation  yields that we need to prove Doeblin condition for the compact $K=[0, 2^{-1}]$, only.

On the segment $[0, 2^{-1}]$, the  intensity of a jump to a point $x_k ,k\geq 1$
is constant and equals $(1-p)p^{k-1}$. If the starting point $x$ belongs to this segment,
 then the process spends inside this segment  a random time that has exponential distribution with intensity 1.
  As soon as  the process jumps to $x_2$, it moves  with the constant speed $a=-1$ and does not have any jumps
    up to any time moment   $t\leq x_2-1$. This means that, for  $t\leq x_2-1$, $x\in [0, 2^{-1}]$,
$$
P_t(x,dy)\geq (1-p)p P(\eta\leq t,\eta\in dy+t-x_2)=(1-p)pe^{-y-t+x_2}\1_{y\leq t}dy,
$$
where $\eta$ denotes an exponential random variable with intensity 1. Therefore, for $T=x_2-1$
$$
\sup_{x,x'\in [0,2^{-1}]}\|P_T(x_1,\cdot)-P_T(x_2,\cdot)\|_{var}\leq 2-(1-p)p\int_0^Te^{-y+1}dy<2,
$$
which gives the Doeblin condition on $[0, 2^{-1}]$.

\emph{Proof of statement (2): sufficiency.} Let us determine the dual process $X^*$.  Note that under assumption  $\sup_k(x_{k+1}-x_k)<+\infty$ one has  $\sum_{k\geq 1}p^{k}e^{\alpha x_k}<+\infty$ for sufficiently small $\alpha>0$. Hence,  by  statement (1), the invariant measure $\pi$ is unique.

Denote $I_0=[0,x_1),I_k=(x_k, x_{k+1}), k\geq 1$. The invariant measure $\pi$ is determined by the relations
$$
\int_\XX \Af fd\pi=0,\quad f\in Dom(\Af)\eqdef\{f:\hbox{ $f$ and
$\Af f$ are bounded}\}.
$$
Taking in this relation $f\in C^1$ with supp$\, f\subset I_k$, we get that $\pi|_{I_k}$ has a density $\rho_k$. In addition, this density is
constant for  $k\geq 1$ and has the form
$$
C[a(x)]^{-1}\exp\left[\int_1^x{\theta(y)\over a(y)}dy\right]
$$
for $k=0$. On the other hand, taking $f\in C^1$ with supp$\, f\subset I_{k-1}\cup I_k$, we obtain that $\pi(\{x_k\})=0, k\geq 1$, and
$$
\rho_k(x_k)=p\rho_{k-1}(x_k),\quad k\geq 1.
$$
These relations and normalizing condition $\pi(\XX)=1$ determine the invariant measure $\pi$ uniquely.

As soon as $\pi$ is determined, one can find the transition probability for the dual process using the
relations
$$
\int_{A}P_t^*(x,B)\pi(dx)=\int_{B}P_t(x,A)\pi(dx), \quad A,B\in \Bf(\XX).
$$
Without a detailed exposition of this standard step, we just give the description of the dual process. Its dynamics also  contains two components. The deterministic component is given by the ODE $dx=a(x)dt$. When the process comes to one of the points $x_k, k\geq 1$, it can either continue its move or make a jump into the segment $[0,1]$. The probability of a jump is equal $(1-p)$, and the distribution of the position of the process after a jump has the density
$$
{\theta(x)\over a(x)}\exp\left[\int_1^x{\theta(y)\over a(y)}dy\right].
$$

Consider the function $\phi^*(x)=E_xe^{\alpha \tau^*}$, where $\tau^*$ is the hitting time of the segment $[0,1]$ by the dual process $X^*$.
For $x\leq 1$, $\phi^*(x)=1$. If the starting point is $X_0^*=x> 1$, the process $X_t^*$ moves with the constant speed $1$ and, at every point $x_k,k\geq 1$, gets a chance  to jump into the segment $[0,1]$ with probability $(1-p)$. Hence
for $x>1$ one has
$$
\phi^*(x)=\sum_{k=K(x)}^\infty(1-p)p^{k-K(x)}e^{\alpha(x_k-x)},
$$
where $K(x)=\inf\{k: x_k>x\}$. Therefore, for  $\alpha>0$ small enough, the function $\phi^*$ is bounded:
$$
\ba \phi^*(x)&\leq \sum_{k=K(x)}^\infty
(1-p)p^{k-K(x)}e^{\alpha(k-K(x)+1)\sup_k(x_{k+1}-x_k)}\\
 &\leq \sum_{k=1}^\infty (1-p)p^{k-1}e^{\alpha k\sup_k(x_{k+1}-x_k)}<+\infty.
 \ea
$$
Like it was done in the proof of statement (1), one can verify that $X^*$ satisfies the local Doeblin condition.   Then by Proposition \ref{p22}
we get that $X^*$ admits an exponential $\phi^*$-coupling. By Corollary \ref{c43}, we get the required statement.

\begin{rem} It can be  verified that the initial process $X$ does not admit an exponential $\phi$-coupling for any bounded $\phi$. One can say that,  in the example in the discussion, the ergodic properties of the dual process $X^*$ are better
 than those of the process $X$ itself. On the other hand, $L_p$ estimates for the process $X$ are equivalent to  $L_q$ estimates for the process $X^*$ ($p^{-1}+q^{-1}=1$). Hence, the  one  interested in $L_p$ rates can choose to start the investigation either from  $X$ or from  $X^*$ depending on their ergodic properties. This is exactly what we have done in our proof. Another possibility is provided by Theorem \ref{t43},
 where ergodic properties  of $X$ and $X^*$ are exploited jointly. We will use this possibility in section 5 below.

\end{rem}

\emph{Proof of statement (2): necessity.} Let
  $X$ satisfy $SG_p(\gamma)$ with some $p\in(1,+\infty), \gamma>0$. Then $0$ is a resolvent point for the restriction of $\{T_t\}$ to $L_p^0$, and therefore there exists $C_1\in\ax$ such that
\be\label{b1}
\mathop{\lim\sup}_{\lambda\to 0+}\left|\int_0^\infty\int_\XX e^{-\lambda t}T_tf(x)g(x)\pi(dx)\, dt\right|\leq C_1\|f\|_p\|g\|_q,\quad  f\in L_p^0,\,g\in L_q^0.
\ee
For every given $Q>0$, (\ref{b1}) also holds true with $f$ replaced by $T_Q f$. Then easy transformation gives
\be\label{b2}
\mathop{\lim\sup}_{\lambda\to 0+}\left|\int_0^Q\int_\XX e^{-\lambda t}T_tf(x)g(x)\pi(dx)\, dt\right|\leq 2C_1\|f\|_p\|g\|_q,\quad  Q\in\ax,\,f\in L_p^0,\,g\in L_q^0.
\ee
Denote $d_k=x_{k+1}-x_k$, $y_k=x_k+{1\over 4}d_k, z_k=x_k+{1\over 2}d_k$ and put
$$f_k=g_k=\1_{(x_k,y_k)}-\1_{(y_k,z_k)}, \quad k\geq 1.
$$
For $t\leq {1\over 2}d_k$, we have $T_tf_k(x)=f_k(x-t)=\1_{(x_k+t,y_k+t)}(x)-\1_{(y_k+t,z_k+t)}(x)$. Recall that the invariant measure  $\pi$ has a positive constant density $\rho_k$ on every segment $I_k=(x_k,x_{k+1})$. Then straightforward calculations show that
$$
\int_\XX T_tf(x)g(x)\pi(dx)\geq 4^{-1}d_k\rho_k,\quad t\leq  4^{-1}d_k.
$$
On the other hand, $\|f_k\|^p=\|g_k\|^q={1\over 2}d_k\rho_k$. Therefore, inequality (\ref{b2}) with $Q= {1\over 2}d_k$ gives the estimate
$$
8^{-1}d_k^2\rho_k\leq 2C_1d_k\rho_k,\quad k\geq 1,
$$
which implies  that the sequence $\{d_k=x_{k+1}-x_k\}$ is bounded.

\emph{Proof of statement (3).} For a fixed $k\geq 1$, consider the function $f_k$ introduced in the previous proof. We have
$$
\|T_tf_k\|^2_2=\|f_k(\cdot -t)\|^2_2=\|f_k\|^2_2,\quad t\leq 2^{-1}d_k.
$$
But under the Poincar\'e inequality (\ref{33}) one should have
$$
\|T_tf\|^2< \|f\|^2_2,\quad t>0, \quad f\in L_2^0, f\not=0.
$$
Therefore, for the process $X$  the Poincar\'e inequality fails.

\section{Solutions to SDE's with jump noise}

In this section we apply the general results of Section 3 to  solution to  SDE of the type
\be\label{01} d X(t)=a(X(t))dt+\int_{\|u\|\leq 1} c(X(t-),u)\tilde
\nu(dt,du)+\int_{\|u\|>1} c(X(t-),u) \nu(dt,du).
 \ee
Here $\nu$ is a Poisson point measure on $\ax\times \Re^d$ with the intensity measure $dt\mu(du)$,
 $\mu$ is the corresponding  L\'evy measure,  $\tilde \nu(dt,du)=\nu(dt,du)-dt\mu(dt)$
  is the compensated point measure, and coefficients $a,c$ satisfy standard conditions
  sufficient for existence and uniqueness of a strong condition (e.g. local Lipschitz and linear growth conditions).

We start the discussion mentioning that, for the process $X$
defined by (\ref{01}), efficient tools to provide existence of an
exponential $\phi$-coupling are available. In \cite{Kul09}, it was
demonstrated that, for such processes, the local Doeblin condition
can be verified efficiently. This condition follows from
appropriate support condition (\cite{Kul09}, condition
\textbf{S}) and a (partial) continuity in variation of the law of
the solution to SDE w.r.t. initial value. The latter property
means that there exists a subset $\Omega'$ of the initial
probability space $\Omega$ such that $P(\Omega')>0$ and the law
of the solution conditioned by $\Omega'$ is continuous in total
variation norm. This property holds under a  non-degeneracy
condition formulated in terms of the random point measure
(\cite{Kul09}, condition \textbf{N}), and the main tool in its
proof is a certain version of a stochastic calculus of variations
for SDE's with jumps. We do not give a detailed overview here,
referring interested reader to \cite{Kul09}.

At the same time, the Lyapunov-type condition for solutions to
SDE's of the type (\ref{01}) is quite transparent
 (see \cite{Mas07}, \cite{Kul09} and discussion therein). Therefore, for solutions to SDE's with
  jump noise, one can prove existence of an exponential $\phi$-coupling using Proposition \ref{p21}.

  However, solutions to SDE's with jump noise, typically, are not time-reversible. The example given in section 4 indicates that,   to investigate $L_p$ convergence rates and spectral properties for a time-irreversible Markov process,   it may be insufficient to have an exponential $\phi$-coupling for the process itself. In general, an analysis of the ergodic properties of the
 the dual processes is also required. In this section, we provide  such an analysis and give sufficient condition for the process $X$ defined by (\ref{01}) to possess a spectral gap property.

 In order to keep exposition reasonably short, we restrict our considerations by a particular, but important class of one-dimensional \emph{L\'evy driven Ornstein-Uhlenbeck processes}; that is, solutions to (\ref{01}) with lineal drift  and additive jump noise. Henceforth,
 in the rest of this section, $X$ is a real-valued process solution to SDE
\be\label{91} dX_t=-a X_t\,dt+dZ_t, \ee where $a>0$ and
$Z_t=\int_0^t\int_{|u|\geq
1}u\nu(ds,du)+\int_0^t\int_{|u|<1}u\tilde \nu(ds,du)$ is a {\it
L\'evy process}.

Ergodic properties for  L\'evy driven Ornstein-Uhlenbeck
processes are well studied. It is
 known that a L\'evy driven Ornstein-Uhlenbeck process $X$ is ergodic if and only if
 $\int_{|u|\geq 1}\ln |u|\mu(du)<+\infty$ (see \cite{SY84}).
   Sufficient conditions for exponential ergodicity for $X$ are also available (see \cite{Mas07} and references therein).
 Our intent is to establish a spectral gap property for the (unique) stationary version of $X$. We give one sufficient condition of that type. Remark that this condition is not strongest possible and allows various generalizations; see Remarks \ref{r91} and \ref{r92} after the proof of Theorem \ref{p91}.

\begin{thm}\label{p91} Assume that

1) $\mu(\Re^-)=\mu(\Re^+)=\infty$;

2) $\mu$ is supported by a bounded set;

3) $\int_{|u|\leq 1}|u|\mu(du)<+\infty$.

\noindent Then, for every $p>1$, process $X$ satisfies $GB_p(\gamma)$  for sufficiently small $\gamma>0$.
\end{thm}

\begin{proof}
We will prove that $X$ and  $X^*$  admit an exponential $\phi$-coupling with the same $\phi(x)$; then  Theorem \ref{t43} would yield the required statement.

It is known (see \cite{KK}, Proposition 2.1)  that under condition
1) the invariant distribution $\pi$ admits a $C^\infty$ density;
denote this density by  $\rho$. In the sequel, we need the
following asymptotic result.  Denote
$$
M_1(\xi)=\int_{\Re}u(e^{\xi u}-1)\mu(du),\quad M_2(\xi)=\int_{\Re}u^2 e^{\xi u}\mu(du), \quad \xi\in \Re,
$$
$$
\Mf_k(\xi)=\int_0^\infty M_k(e^{-a s}\xi)\, ds\quad \xi\in \Re, \,k=1,2.
$$
Clearly,
$$
{d\over d\xi}\Mf_1(\xi)=\Mf_2(\xi)>0
$$
and condition 1) yields $\Mf_1(\xi)\to \pm\infty, \xi\to
\pm\infty$. Then for every $x\in \Re$ there exists unique solution
$\xi=\xi(x)$ to the equation
$$
\Mf_1(\xi)=x.
$$

\begin{prop}\label{p92} (\cite{KK}, Theorem 7.1) Under conditions 1) and 3) of Theorem \ref{p91},
$$
\left[{\rho(x+y)\over \rho(y)}\right] e^{y\xi(x)}\to 1,\quad x\to \infty
$$
uniformly by $y\in Y$ for every bounded set $Y\subset \Re$.
\end{prop}

 We  proceed with the proof of the theorem.  Our first step is to specify the dual process $X^*$. Since this step is quite standard, we  sketch the argument and omit  technical details.

Every $f\in C^1$ with at most polynomial growth of its derivative belongs to the domain of the extended generator $\Af$ and
$$
\Af f(x)=-a xf'(x)+\int_\Re[f(x+u)-f(x)]\mu(du).
$$
 We  have
$$
\int_\Re\Af fd\pi=0
$$
for every such $f$. This yields that  the invariant density
$\rho$ satisfies the relation \be\label{92} a
x\rho'(x)+a\rho(x)+\int_\Re[\rho(x-u)-\rho(x)]\mu(du)=0 \ee The
formally adjoint  operator to $\Af$ is given by the formula
$$
\Af^*f(x)=\rho^{-1}(x)\Big({d\over dx}[ax\rho(x)f(x)]+\int_\Re[f(x-u)\rho(x-u)-f(x)\rho(x)]\mu(du)\Big).
$$
This relation combined with (\ref{92}) provides that the extended generator of the dual process $X^*$ is given by
$$
\Af^*f(x)=a xf'(x)+\int_\Re\left[(f(x-u)-f(x)){\rho(x-u)\over \rho(x)}\right]\mu(du).
$$
Note that the domain of $\Af^*$ is yet to be determined. It can be verified by additional
investigation that $\rho(x)>0, x\in\Re$, but this would lead to unnecessary complication of the proof.
For our needs, it is sufficient  to refer to Proposition \ref{p92} which implies that $\rho(x)>0$ outside
some segment $[-I,I]$. Furthermore  it is easy to verify that, for every $\eps>0$,
$$
\Mf_1(\xi)e^{-(\sigma_*+\eps)|\xi|}\to 0,\quad \xi \to \infty
$$
with  $\sigma_*\eqdef\inf\{\sigma:\mu(|u|>\sigma)=0\}$ (see
\cite{KK}, Example 5.2), and consequently
$$
\sup_{|u|\leq \sigma_*}{\rho(x-u)\over \rho(x)}\leq C(1+|x|^p),
$$
where $C,p>0$ are some constants. Hence every $f\in C^1$  with at most polynomial growth of its derivative, being constant on $[-I-\sigma_*,I+\sigma_*]$, belongs to the domain of $\Af^*$.

Consider $\phi\in C^1$ such that $\phi\geq 1, \phi$ is constant on $[-I-\sigma_*,I+\sigma_*]$, and $\phi(x)=|x|$ for $|x|$ large enough.
Then for  $|x|$ large enough
$$
\Af\phi(x)=-a x\, \mathrm{sign}\,x+\int_{-\sigma_*}^{\sigma_*}u\mu(du)\leq -{a\over 2}|x|=-{a\over 2}\phi(x);
$$
that is, $\phi$ satisfies the Lyapunov-type condition (\ref{lyap}) w.r.t. the process $X$. On the other hand,
from Proposition \ref{p92} we get that, for every $\delta>0$, there exist $C>0$ such that
\be\label{93}
\begin{aligned}
\Af^*\phi(x)&\leq ax\mathop{\mathrm{sign}} x-(1-\delta) \int_{\Re}(|x-u|-|x|)e^{-u\xi(x)}\mu(du)\\
&=\mathop{\mathrm{sign}} x\left[ax-(1-\delta)M_1(\xi(x))\right]=|x|[a-(1-\delta)x^{-1}M_1(\xi(x))].
\end{aligned}
\ee It can be verified easily (e.g. see the proof of Theorem 7.1)
that for every $\sigma\in(0,1)$
$$
M_1(\sigma\xi)[M_1(\xi)]^{-1}\to 0,\quad \xi \to \infty
$$
and consequently
$$
\Mf_1(\xi)[M_1(\xi)]^{-1}\to 0,\quad \xi \to \infty.
$$
Therefore $x^{-1} M_1(\xi(x))\to+\infty, x\to \infty$ because $\Mf_1(\xi(x))=x$. This and (\ref{93}) yield that $\phi$ satisfies the Lyapunov-type condition (\ref{lyap}) w.r.t. the process $X$.

For the process $X$, the local Doeblin condition holds; we have
already mentioned that one can derive this condition using
results of \cite{Kul09}. In particular, in the case under
consideration one can deduce the  local Doeblin condition from
Theorem 1.3 \cite{Kul09} using literally the same arguments with
those given in the proof of Proposition 0.1 \cite{Kul09}.

On the other hand, $X^*$ is not a solution to SDE of the type
(\ref{01}). It is a \emph{process with non-constant rate of
jumps}, and  such processes were not considered in \cite{Kul09}.
Henceforth, one  can not deduce the local Doeblin condition for
$X^*$ from the results of \cite{Kul09}. However,  the stochastic
calculus of variations that provides (partial) continuity in
variation is available for the processes with non-constant rate
of jumps as well, see \cite{Kul08}, and the main  results from
\cite{Kul09} can be extended for such  processes   without
principal changes. In particular, one can prove the local Doeblin
condition for $X^*$ following the proof of  Proposition 0.1
\cite{Kul09} and using within this proof Theorem 4.2 \cite{Kul08}
instead of Theorem 1.3 \cite{Kul09}.

Now we apply Proposition \ref{p21} twice, and get that both $X$
 and $X^*$ admit an exponential $\phi$-coupling. Applying  Theorem \ref{t43} completes the proof.
\end{proof}

\begin{rem}\label{r91}\emph{(On the class of equations).}  In Theorem
 \ref{p91}, we restrict our consideration by the linear SDE's with jump noise.
   The only point in the proof where this structural assumption was used substantially
   -- the Lyapunov-type condition for $\phi$ w.r.t. the dual process -- is based on the estimates
   for the ratio ${\rho(x+y)\over \rho(x)}.$ In \cite{KK}, these estimates are obtained via harmonic analysis
   arguments, and here is the point where the linear structure of SDE's under investigation is substantial.
    However, this is not the only possible technique. Supposedly, using `stochastic calculus
     of variations' tools similar to those given in \cite{K06}, Section 6, one can extend such estimates to
      non-linear SDE's with jump noise as well, and then give an extension of Theorem \ref{p91}
      to this more general class of equations. We postpone such a generalization to  a further publications.

\end{rem}

\begin{rem}\label{r92}\emph{(On conditions).} Conditions 1) and 2) of Theorem \ref{p91} come
from \cite{KK} Theorem 7.1. The first condition is rather mild,
and the second one allows a wide field of modifications. For
instance, it can be replaced by an appropriate condition on the
`exponential tails' of the L\'evy measure $\mu$ (see \cite{KK},
Proposition 6.1 and discussion before Theorem 7.1).
 On the other hand, condition 3), though not used explicitly, is crucial in our framework. Without this
 condition one can not apply Theorem 4.2 \cite{Kul08} which is required  to get the local Doeblin condition
 for the dual process.
\end{rem}

\section{Exponential moments for hitting times under Poincar\'e inequality}

The results of  section 3 allows one to establish spectral gap
property for a given process $X$ in the following way: first,
prove the local Doeblin condition to hold true; second, find {\it
some} $\phi$ such that the recurrence conditions 1) -- 3) from
Theorem \ref{p21} holds true; then, if $X$ is time-irreversible,
repeat this procedure for the dual process $X^*$; and, finally,
deduce the required property using Theorems \ref{t43} --
\ref{t44}.

We have already mentioned that the local
Doeblin condition is straightforward, and can be verified
efficiently for important  classes of  processes like  diffusions or solutions to SDE's with jump noise. The
second part in the framework outlined above -- the recurrence conditions -- looks less
transparent since there is a lot of freedom in the choice of
$\phi$.  Proposition \ref{p22}, in fact,  reduces such a choice  to the
class of functions of the form $\phi(x)=E_xe^{\alpha\tau_K}$. In
this section we demonstrate that this  reduction well
corresponds to the matter of the problem.

Considerations of this section are mainly motivated by the paper
\cite{Mat97}, where the relation  between the family of certain
weak versions of the Poincar\'e inequalitiy, on one hand, and the
moments of the hitting times
$$
\tau_K=\inf\{t: X_t\in K\},
$$
on other hand, is investigated.

In what follows, we suppose an invariant measure $\pi$ for the
process $X$ to be fixed, and consider the Dirichlet form $\Ef$ on
$L_2\eqdef L_2(\XX,\pi)$ corresponding to the process $X$ (see
section 2.2). The form $\Ef$ is supposed to be regular; that is,
the set $Dom(\Ef)\cap C_0(\XX)$ is claimed to be dense both in
 $Dom(\Ef)$ w.r.t. the norm $\|\cdot\|_{\Ef,1}$ and in $C_0(\XX)$
 w.r.t. uniform convergence on a compacts ($C_0(\XX)$ is the set of continuous functions with compact
 supports).  We also assume that the {\it sector condition} holds true:
 $$
\exists D\in\ax:\quad  |\Ef( f,g)|\leq D
\|f\|_{\Ef,1}\|g\|_{\Ef,1},\quad f,g\in Dom(\Ef).
 $$

 It is well-known (see the discussion in
Introduction to \cite{Mat97} and references therein) that the
hitting times  $\tau_K$ have natural application in the
probabilistic representation for the family of {\it
$\alpha$-potentials} for the Dirichlet form $\Ef$. The
$\alpha$-potential, for given $\alpha>0$ and closed $K\subset
\XX$, is defined as the function $h_\alpha^K\in Dom(\Ef)$ such
that $h_\alpha^K=1$ quasi-everywhere on $K$, and
$\Ef(h_\alpha^K,u)=-\alpha(h_\alpha^K,u)$ for every
quasi-continuous function $u\in Dom(\Ef)$ such that $u=0$
quasi-everywhere on $K$. On the other hand,
$$
h_\alpha^K(x)=E_xe^{-\alpha \tau_K}, \quad x\in\XX.
$$

It is a straightforward corollary of the part (i) of the main
theorem from \cite{Mat97} that, if $X$ possesses $PI(\gamma)$ with
some $\gamma>0$, then $E_\pi \tau_K<+\infty$  for every $K$ with
$\pi(K)>0$  (here and below, $E_\pi\eqdef \int_\XX
E_x\,\pi(dx)$). We will prove the following  stronger version of
this statement.

\begin{thm}\label{t51} Assume $X$ possess $PI(\gamma)$ with some
$\gamma>0$. Then for every closed set $K\subset \XX$ with
$\pi(K)>0$
$$
E_\pi e^{\alpha\tau_K}<+\infty,\quad  \alpha<{\gamma\pi(K)\over
2}.
$$
Moreover, the function $h_{-\alpha}^K(x)\eqdef E_x
e^{\alpha\tau_K}, x\in\XX$ possesses the following properties:

a) $h_{-\alpha}^K\in Dom(\Ef)$ and $h_{-\alpha}^K=1$ on $K$;

b) $\Ef(h_{-\alpha}^K,u)=\alpha(h_{-\alpha}^K,u)$ for every
quasi-continuous function $u\in Dom(\Ef)$ such that $u=0$
quasi-everywhere on $K$.
\end{thm}

 \begin{proof} We assume $K$ to be fixed and omit the
 respective index in the notation, e.g. write $\tau$ for $\tau^K$ and $h_\alpha$ for $h_\alpha^K$.  For $z\in \CC$ with
 $\mathrm{Re}\,z>0$, define respective $z$-potential:
 $$
 h_z(x)=E_xe^{-z \tau}, \quad x\in\XX.
 $$

Denote by $H_\Ef$ the  $Dom(\Ef)$ considered as a Hilbert space
with the scalar product $(f,g)_{\Ef,1}\eqdef
(f,g)_{L_2}+\Ef(f,g)$. The following lemma shows  that
$\{h_z,\mathrm{Re}\, z>0\}$  can be considered as an analytical
extension of the family of $\alpha$-potentials
$\{h_\alpha,\alpha>0\}\subset H_\Ef$ that, in addition,  keeps the properties of
this family.

\begin{lem}\label{l51} 1) The function $z\mapsto h_z$ is analytic as a function
taking values in the Hilbert space $H_\Ef$.

2) For every $z$ with $\mathrm{Re}\, z>0$, the following
properties hold:

(i) $h_z=1$ quasi-everywhere on $K$;

(ii) $\Ef(h_z,u)=-z(h_z,u)$ for every quasi-continuous function
$u\in Dom(\Ef)$ such that $u=0$ quasi-everywhere on $K$.
\end{lem}

\begin{proof}
 Denote $h_z^m(x)=(-1)^m  E_x\tau^m e^{-z \tau}, x\in\XX, m\geq 1$. One can verify easily that, for every $m\in\NN$,
 $$
 {d^m\over dz^m} h_z=h_z^m
$$
on the set $\{z: \mathrm{Re}\,z>0\}$, with the function $z\mapsto
h_z$ is considered as a function taking values in $L_2$. In
addition,
$$
\|h_z^m\|^2_2\leq E_\pi \left|\tau^m e^{-z \tau}\right|^2=E_\pi
\tau^{2m}e^{-2\tau\mathrm{Re}\, z}\leq {(2m)!\over
(2\mathrm{Re}\,z)^{2m}},
$$
since ${(2\tau\mathrm{Re}\, z)^{2m}\over (2m)!}\leq
e^{2\tau\mathrm{Re}\, z}$. Therefore, \be\label{52}
{\|h^m_z\|_2\over m!}\leq \sqrt{C_{2m}^m\over
2^{2m}}(\mathrm{Re}\, z)^{-m}<(\mathrm{Re}\, z)^{-m},\quad m\in
\NN, \ee and hence the function
$$
\{z: \mathrm{Re}\,z>0\}\ni z\mapsto h_z\in L_2
$$
is analytic.

 For every
$\alpha,\alpha'>0$ we have $h_\alpha-h_{\alpha'}=0$
quasi-everywhere on $K$. Hence \be\label{51}\ba
\Ef(h_\alpha-h_{\alpha'},h_\alpha-h_{\alpha'})&=\Ef(h_\alpha,h_\alpha-h_{\alpha'})-\Ef(h_{\alpha'},h_\alpha-h_{\alpha'})\\
&=-\alpha(h_\alpha,h_\alpha-h_{\alpha'})+\alpha'(h_{\alpha'},h_\alpha-h_{\alpha'})\\
&=(\alpha'-\alpha)(h_{\alpha'},h_\alpha-h_{\alpha'})+ \alpha
(h_{\alpha'}-h_{\alpha},h_\alpha-h_{\alpha'}). \ea \ee Therefore,
for a given $\alpha>0$ and $\alpha'\to \alpha$, the family
$\{{h_{\alpha'}-h_{\alpha}\over \alpha'-\alpha}\}$ is bounded in
$H_\Ef$, and thus is weakly compact in $H_\Ef$. On the other
hand, this family converges to $h_\alpha^1$ in $L_2$. This yields
that the function  $(0,+\infty)\in\alpha\mapsto h_\alpha\in
H_\Ef$ is differentiable in a weak sense, and  $h^1_\alpha$
equals its (weak) derivative at the point $\alpha$.

We have $h_\alpha^1=0$ quasi-everywhere on $K$, since
$$
h_\alpha(x)=1\Leftrightarrow e^{-\alpha \tau}=1 \
P_x-\hbox{a.s.}\Leftrightarrow  \tau=0 \
P_x-\hbox{a.s.}\Leftrightarrow h_\alpha^1(x)=0.
$$
In addition, since $h_\alpha^1$ is a weak derivative of
$h_\alpha$, we have
$\Ef(h_\alpha^1,u)=-(h_\alpha,u)-\alpha(h_\alpha^1,u)$ for every
quasi-continuous function $u\in Dom(\Ef)$ such that $u=0$
quasi-everywhere on $K$.  Now, repeating the same arguments, we
get by induction that, for every $m\geq 1$, the function
$(0,+\infty)\in\alpha\mapsto h_\alpha\in H_\Ef$ is $m$ times
weakly differentiable, $h_\alpha^m$ is the corresponding weak
derivative of the $m$-th order, and the following properties hold:

(i$^m$) $h_\alpha^m=0$ quasi-everywhere on $K$;

(ii$^m$)
$\Ef(h_\alpha^m,u)=-(h_\alpha^{m-1},u)-\alpha(h_\alpha^m,u)$ for
every quasi-continuous function $u\in Dom(\Ef)$ such that $u=0$
quasi-everywhere on $K$.

Property (ii) with $u=h_\alpha^m$ and estimate (\ref{52}) yield
that, for a given $\alpha$, series
$$
H_z\eqdef h_\alpha+\sum_{m=1}^\infty{z^m\over m!} h_\alpha^m\in
H_\Ef
$$
converge in the circle $\{|z-\alpha|<\alpha\}$. The sum is a
weakly analytic $H_\Ef$-valued function, and hence is  analytic
(\cite{Rud73}, Theorem 3.31). On the other hand, the same series
converge in $L_2$ to $h_z$. This yields that $h_z=H_z$ in the
circle $\{|z-\alpha|<\alpha\}$. By taking various
$\alpha\in(0,+\infty)$, we get that the function $z\mapsto h_z$
is  an $H_\Ef$-valued analytic function inside the angle
$\Df_1\eqdef\{z:\mathrm{Re}\, z> |\mathrm{Im}\,z|\}$. In
addition, properties (i$^m$), (ii$^m$) of the $m$-th coefficients
of the series ($m\geq 1$) provide that $h_z$ satisfy (i),(ii)
inside the angle.

Now, we complete the proof using the following iterative
procedure. Assume that the function $z\mapsto h_z\in H_\Ef$ is
analytic in some domain $\Df\subset \{z:\mathrm{Re}\,z >0\}$ and
satisfy (i),(ii) in this domain. Then the same arguments with
those used above show that, for every $z_0\in \Df$, the domain
$\Df$ can be extended to
$\Df'\eqdef\Df\cup\{z:|z-z_0|<\mathrm{Re}\,z_0\}$ with the
function $z\mapsto h_z$ still being analytic in $\Df'$ and
satisfying (i),(ii) in the extended domain. Therefore, we prove
iteratively that the required statement holds true in every angle
$\Df_k\eqdef\{z:\mathrm{Re}\, z> {1\over k} |\mathrm{Im}\,z|\}$.
Since $\cup_k\Df_k=\{z:\mathrm{Re}\,z>0\}$, this completes the
proof.\end{proof}

 Next, we consider "$\psi$-potentials" that correspond to   functions $\psi:\ax\to \Re$.
Denote
$$h_\psi(x)=E_x\psi(\tau),\quad x\in\XX.
$$
The following statement is an appropriate modification of the
inversion formula for the Laplace transform.

\begin{lem}\label{l52} Let $\psi\in C^2(\Re)$ have a compact support and $\mathrm{supp}\,\psi\subset [0,+\infty)$.
Denote  $\Psi(z)=\int_\Re e^{zt}\psi(t)\, dt,$ $z\in\CC$.

 The function
$h_\psi$ belongs to $H_\Ef$ and admits integral representation
\be\label{53} h_\psi={1\over 2\pi i
}\int_{\sigma-i\infty}^{\sigma+i\infty}\Psi(z)h_z\,dz, \ee where
$\sigma>0$ is arbitrary, and the integral is well defined as an
improper Bochner integral of an $H_\Ef$-valued function.
\end{lem}

\begin{proof} First, let us show that the integral in the right
hand side of (\ref{53}) is well defined. We have by condition (ii)
of Lemma \ref{l51} that
$$
\Ef(h_z,h_z)=\Ef(h_z,h_z-1)=-z(h_z,h_z-1).
$$
For any $z$ with $\mathrm{Re}\,z>0$, we have $|h_z(x)|\leq E_x
e^{-\tau\mathrm{Re}\,z}\leq 1$, and thus $|h_z(x)-1|\leq 2$.
Hence,
$$
\|h_z\|_{H_\Ef}= \sqrt{\|h_z\|_2^2+\Ef(h_z,h_z)}\leq
\sqrt{1+2|z|}.
$$
On the other hand, for $\psi$ satisfying conditions of the lemma,
$$
z^2\Psi(z)=\int_0^\infty e^{zt}\psi''(t)\, dt,\quad
|z^2\Psi(z)|\leq \int_0^\infty e^{t\mathrm{Re}\,z}|\psi''(t)|\,
dt.
$$
Thus, on the line $\sigma+i\Re\eqdef\{z:\mathrm{Re}\, z=\sigma\}$,
the function $z\mapsto \Psi(z)h_z\in H_\Ef$ admits the following
estimate:
$$
\|\Psi(z)h_z\|_{H_\Ef}\leq C|z|^{-{3\over 2}},
$$
and therefore it is integrable on $\sigma+i\Re$. Denote by
$g_\psi\in H_\Ef$ corresponding integral. In order to prove that
$h_\psi=g_\psi$, it is sufficient to prove that $h_\psi$ and
$g_\psi$ coincide as elements of $L_2$. Hence, we have reduced the
proof of the lemma to verification of the following "weak
$L_2$-version" of (\ref{53}): \be\label{54} \int_\XX h_\psi v\,
d\pi={1\over 2\pi i}\int_\XX
\int_{\sigma-i\infty}^{\sigma+i\infty}\Psi(z)h_z(x)v(x)\,dz\pi(dx),\quad
v\in L_2. \ee Recall that $h_z(x)=E_xe^{-z\tau}$, and hence the
right hand side of (\ref{54}) can be rewritten to the form
$$
{1\over 2\pi i}\int_\XX
\int_{\sigma-i\infty}^{\sigma+i\infty}E_x\Psi(z)
e^{-z\tau}v(x)\,dz\pi(dx)={1\over 2\pi i}\int_\XX E_x
\int_{\sigma-i\infty}^{\sigma+i\infty}\Psi(z)
e^{-z\tau}v(x)\,dz\pi(dx).
$$
Here, we have changed the order of integration using Fubini's
theorem. This can be done, because $|\Psi(z)|\leq C|z|^{-2}$, and
therefore
$$
E_x \int_{\sigma-i\infty}^{\sigma+i\infty}|\Psi(z) e^{-z\tau}|\,
dz =h_\sigma(x) \int_{\sigma-i\infty}^{\sigma+i\infty}|\Psi(z)|\,
dz\leq C h_\sigma(x).
$$
The function $\Psi$ is the (two-sided) Laplace transform for
$\psi$, up to the change of variables $p\mapsto -z$. We write the
inversion formula for the Laplace transform in the terms of
$\Psi$ and, after the change of variables, get
$$
\psi(t)={1\over 2\pi
i}\int_{-\sigma-i\infty}^{-\sigma+i\infty}e^{pt}\Psi(-p)\,
dp={1\over 2\pi
i}\int_{\sigma-i\infty}^{\sigma+i\infty}e^{-zt}\Psi(z)\, dz,
\quad t\in \ax.
$$
Hence, the right hand side of (\ref{54}) is equal
$$
\int_\XX E_x\psi(\tau)v(x)\pi(dx)=\int_\XX h_\psi v\, d\pi,
$$
that proves (\ref{54}).
\end{proof}

\begin{cor}\label{c51}  Let $\psi\in C^3(\Re)$ and $\mathrm{supp}\,\psi'\subset [0,+\infty)$.
Then $h_\psi\in Dom(\Ef)$ and   \be\label{55}
\Ef(h_\psi,u)=(h_{\psi'},u) \ee for every $u\in Dom(\Ef)$ such
that $u=0$ quasi-everywhere on $K$.
\end{cor}
\begin{proof} Assume first that $\int_{\ax} \psi'(x)\, dx=0$.  Then both $\psi$ and $\psi'$ satisy conditions of Lemma \ref{l52}.
We have $\tilde \Psi(z)\eqdef\int_\Re e^{zt}\psi'(t)\,
dt=-z\Psi(z)$. Hence, from the representation (\ref{53}) for
$h_\psi$ and $h_{\psi'}$ and relation
$\Ef(h_z,u)=-z(h_z,u),\mathrm{Re}\, z>0$, we get
$$
\Ef(h_\psi,u)={1\over 2\pi i
}\int_{\sigma-i\infty}^{\sigma+i\infty}\Psi(z)\Ef(h_z,u)\,dz={1\over
2\pi i }\int_{\sigma-i\infty}^{\sigma+i\infty}\tilde
\Psi(z)(h_z,u)\,dz=(h_{\psi'},u).
$$

The general case can be reduced to the one considered above by
the following limit procedure. Since $\mathrm{supp}\,\psi'\subset
[0,+\infty)$, there exist $C\in \Re$ and  $x_*\in \ax$ such that
$\psi(x)=C, x\geq x_*$. Take a function $\chi\in C^3(\Re)$ with
$\mathrm{supp}\,\chi\subset [0,1]$, and put
$$
\psi_t(x)=\psi(x)-C\chi(x-t),\quad x\in\Re, t>x_*.
$$
Then every $\psi_t$ satisfies the additional assumption
$\int_{\ax} [\psi_t]'(x)\, dx=0$, and thus $h_{\psi_t}$ belongs
to $Dom(\Ef)$ and satisfies (\ref{55}). It can be verified easily
that $h_{\psi_t}\to h_\psi, t\to \infty$ in $L_2$ sense. In
addition,
$$
\Ef(h_{\psi_t},h_{\psi_t})=(h_{[\psi_t]'},h_{\psi_t})\to
(h_{[\psi]'},h_{\psi})<+\infty,\quad t\to+\infty
$$
(here, we have used (\ref{55}) with $u=h_{\psi_t}$). This means
that the family $\{h_{\psi_t}\}$ is bounded in $H_\Ef$, and hence
is weakly compact in $H_{\Ef}$. Therefore, $h_{\psi_t}\to h_\psi,
t\to \infty$ weakly in $H_{\Ef}$. Since $h_{[\psi_t]'}\to
h_{\psi'}, t\to \infty$ in $L_2$ sense, (\ref{55}) for $\psi$
follows from (\ref{55}) for $\psi_t$.
\end{proof}
 Now, we are ready to complete the proof of the theorem. Let us  fix
 $\alpha<{\gamma\pi(K)\over 2}$, and construct the family of the functions $\varrho_t, t\geq 1$ that approximate the function
 $\varrho: x\mapsto e^{\alpha x}-1$ appropriately. First, we take
 function
 $\chi\in C^3(\Re)$ such that $\chi\geq 0, \chi'\leq 0, \chi(x)=1,
 x\leq0$, and $\chi(x)=0, x\geq 1$. We put
 $$
 \rho_t(x)=\int_0^x\alpha e^{\alpha y}\chi(y-t)\, dy, \quad
 x\geq 0, t\geq 1.
 $$
 By the construction, the derivatives of the functions $\rho_t, t\geq
 1$ have the following properties:

 a) $[\rho_t]'\geq 0$ and $[\rho_t]'(x)=0, x\geq t+1$;

 b) $[\rho_s]'\leq[\rho_t]', s\leq t$.

 Since $\rho_t(0)=0, t\geq 1$, the latter property yields that $\rho_s\leq \rho_t, s\leq
 t$. In addition,
 $$
 [\rho_t]''(x)= \alpha e^{\alpha x} \chi'(x)+\alpha^2 e^{\alpha x} \chi(x)\leq
 \alpha^2 e^{\alpha x} \chi(x)=\alpha [\rho_t]'(x),
 $$
 since $\chi'\leq 0$. This and relation $[\rho_t]'(0)=\alpha (\rho_t(0)+1)$
 provide
 \be\label{56} [\rho_t]'\leq \alpha(\rho_t+1).
 \ee

At last,  we take  function $\theta\in C^3(\Re)$ such that
$\theta'\geq 0, \theta(x)=0,
 x\leq 0$, and $\theta(x)=1, x\geq 1$. We put
 $$
 \varrho_t(x)=\begin{cases} \theta\left({xt}\right)\rho_t(x),&x\geq
 0\\
  0,& x<0
 \end{cases},\quad t\geq 1.
 $$
We have $\varrho_t\uparrow \varrho, t\uparrow \infty$. In
addition, by (\ref{56}), \be\label{57}
[\varrho_t]'(x)=t\theta'(tx)\rho_t(x)+\theta(tx)[\rho_t]'(x)\leq
t\sup_{y}\theta'(y)\rho_t(t^{-1})+\alpha(\rho_t(x)+1)\leq
\alpha\varrho_t(x)+C \ee with an appropriate constant $C$ (recall
that $t\rho_t(t^{-1})=t\alpha(e^{\alpha t^{-1}}-1)\to \alpha^2,
t\to \infty$).

Every $\varrho_t$ satisfies conditions of Corollary \ref{c51},
and hence $$
\int_{\XX}h_{\varrho_t}^2\,d\pi-\left(\int_{\XX}h_{\varrho_t}\,d\pi\right)^2\leq
{2\over \gamma}\Ef(h_{\varrho_t}, h_{\varrho_t})={2\over
\gamma}(h_{[\varrho_t]'}, h_{\varrho_t})\leq {2\alpha\over
\gamma}(h_{\varrho_t}, h_{\varrho_t})+C\int_\XX h_{\varrho_t}\,
d\pi.$$ Here, we have used subsequently property $PI_2(\gamma)$,
equality  (\ref{56}) with $u=h_{\varrho_t}$, and (\ref{57}).

We have $h_{\varrho_t}=0$ on $K$ because $\varrho_t(0)=0$. Then,
by the Cauchy inequality,
$$\ba
\int_{\XX}h_{\varrho_t}^2\,d\pi-&\left(\int_{\XX}h_{\varrho_t}\,d\pi\right)^2=
\int_{\XX}h_{\varrho_t}^2\,d\pi-\left(\int_{\XX\setminus K
}h_{\varrho_t}\,d\pi\right)^2\\
&\geq (1-\pi(\XX\setminus K))
\int_{\XX}h_{\varrho_t}^2\,d\pi=\pi(K)(h_{\varrho_t},
h_{\varrho_t}).\ea
$$
Therefore,  $$(h_{\varrho_t}, h_{\varrho_t})\leq {2\alpha \over
\gamma \pi(K)}(h_{\varrho_t}, h_{\varrho_t})+C\int_\XX
h_{\varrho_t}\, d\pi,$$ which implies that \be\label{58}
(h_{\varrho_t}, h_{\varrho_t})\leq C\int_\XX h_{\varrho_t}\, d\pi
\ee (recall that $\alpha<{\gamma\pi(K)\over 2}$). One can verify
easily that (\ref{58}) yields that the $L_2$-norms of the
functions $h_{\varrho_t}$ are uniformly bounded. Since
$\varrho_t\uparrow \varrho$, this implies that the function
$$
h_\varrho (x)\eqdef E_xe^{\alpha \tau}-1,\quad x\in\XX
$$
belongs to $L_2$, and $h_{\varrho_t}\to h_\varrho,
 t\to \infty$ in $L_2$. Similarly to the proof of Corollary
\ref{c51}, one can verify that  $\{h_{\varrho_t}\}$ is a
bounded subset in $H_\Ef$, and hence $h_{\varrho_t}\to h_\varrho,
 t\to \infty$ weakly in $H_\Ef$. This proves statement a) of the
 theorem.  In order to prove statement b), we apply (\ref{55}) to
$\psi=\varrho_t$, and pass to the limit as $t\to +\infty$. The
 theorem is proved.\end{proof}

\section{Poincar\'e inequality for diffusions: criterion in the terms of  hitting times}

In this section, we apply our general results  to  diffusion
processes on non-compact manifolds. The Poincar\'e inequality for
diffusions was studied
 extensively by numerous authors. We refer to \cite{RW04}, \cite{Wang00} for various sufficient conditions for this inequality and further references. The main result of this
section -- Theorem \ref{t61} -- is a refinement of Theorem
\ref{t42}, Theorem \ref{t51}, and Proposition \ref{p22}.  It
gives {necessary and sufficient} condition for the Poincar\'e
inequality in the terms of  hitting times of the diffusion
process.

Let $\XX$ be a connected locally compact  Riemannian manifold of
dimension $d$, and $X$ be a diffusion process on $\XX$. On a given
local chart of the manifold $\XX$, the generator of the process
$X$ has the form
$$
A=\sum_{j=1}^d a_j\prt_j+{1\over 2}\sum_{j,k=1}^d
b_{jk}\prt^2_{jk},
$$
where $a=\{a_j\}_{j=1}^d$ and $b=\{b_{jk}\}_{j,k=1}^d$ are the
drift and diffusion coefficients of the process $X$ on this chart,
respectively. We assume the coefficients $a,b$ to be H\"older
continuous on every local chart, and the drift $b$ coefficient to
satisfy ellipticity condition
$$
\sum_{j,k=1}^d b_{jk} v_jv_k\geq c\sum_{j=1}^dv_j^2
$$
uniformly on every compact. Under these conditions, the
transition function of the process $X$ has a positive density
w.r.t. Riemannian volume, and this density  is a continuous
function on $(0,+\infty)\times \XX\times \XX$. One can easily
deduce this from the same statement for diffusions in $\Re^d$
(e.g. \cite{IKO62}) and strong Markov property of $X$. This
implies that $X$ satisfies  the extended Doeblin condition on
every compact subset of $\XX$.

Let $\pi\in \Pf(\XX)$ be an invariant measure for the process $X$
(we assume invariant measure to exist). Denote
by $\Ef$ the Dirichlet form on $L_2(\XX,\pi)$ corresponding to
$X$.

\begin{thm}\label{t61} The following statements are equivalent:

1) the Poincar\'e inequality holds true with some  constant $c$:
$$
\int_\XX |f|^2d\pi-\left|\int_\XX f d\pi\right|^2\leq c
\,\Ef(f,f),\quad f\in Dom(\Ef).$$

2) the process $X$ admits an exponential $\phi$-coupling for some
function $\phi$;

3) for every closed subset $K\subset \XX$ with $\pi(K)>0,$ there
exists $\alpha>0$ such that
$$E_\pi e^{\alpha\tau_K}<+\infty.
$$

\noindent In addition, 1) -- 3) hold true assuming that

3$\,'$) there exists a compact subset $K\subset \XX$ and
$\alpha>0$ such that
$$E_x e^{\alpha\tau_K}<+\infty\quad \hbox{for $\pi$-almost all}\quad x\in X.
$$

\end{thm}
\begin{proof} Implications 2) $\Rightarrow$ 1) and
1) $\Rightarrow$ 3) are proved in Theorems \ref{t42} and
\ref{t51}, respectively. Hence, we need to prove implication
3$\,'$) $\Rightarrow$ 2), only. We will prove it using
Proposition \ref{p22}. In order to simplify exposition, we
consider the case $\XX=\Re^d,$ only. One can easily extend the
proof to the general case by a standard localization procedure.

We take $\tilde \alpha\in (0,\alpha)$ and put
$\phi(x)=E_xe^{\tilde \alpha\tau_K}, x\in\XX$. Let us show that
$\phi$ is locally bounded; that is, condition 2) of Proposition
\ref{p22} holds true with $\alpha$ replaced by $\tilde \alpha$.

Let $x_0\in \Re^d$ and $0<r_0<r_1$ be such that $K\subset
\{x:\|x-x_0\|<r_0\}$. Denote $D=\{x:\|x-x_0\|<r_1\}\setminus K$,
$\theta=\inf\{t:X_t\in \prt D\}$, and  $\mu_{x}(dy)\eqdef
P_x(X_\theta\in dy), x\in D$.

Consider auxiliary function
$$
h(x)=\int_{\prt D}E_{y}e^{\alpha \tau_K}\, \mu_x(dy),\quad x\in D.
$$
This function is $A$-harmonic in $D$, hence it satisfies the
Harnack inequality (see \cite{KS81}). Namely, there exists
$C\in\ax$ such that
$$
h(x_1)\leq Ch(x_2)
$$
for every $y\in D$, and $x_1,x_2\in \{x: \|x-y\|<{1\over
2}\mathrm{dist}(y,\prt D)\}$. On the other hand, by the strong
Markov property of $X$, we have
$$
E_xe^{\alpha\tau_K}=E_x(e^{\alpha\theta}\phi(X_\theta))\geq
E_x\phi(X_\theta)=h(x), \quad x\in D.
$$
Hence, under condition $3')$, $h(x)<+\infty$ for $\pi$-a.a. $x\in
D$. In addition,  $\mathrm{supp}\, \pi=\XX$;  one can easily
verify this fact using positivity of the transition probability
density. Therefore, the function $h$ is bounded on every compact
$S\subset D$.

The function $h$ can be written in the form
$$
h(x)=E_xe^{\alpha\tau_K^\theta},\quad \tau_K^\theta=\inf\{t\geq 0:
X_{t+\theta}\in K\}.
$$
For $x\in D$, we have $\tau_K=\theta+ \tau_K^\theta$ $P_x$-a.s.,
and therefore
$$
E_xe^{\tilde \alpha\tau_K}\leq [E_x(e^{{\alpha\tilde\alpha\over
\alpha-\tilde\alpha}\theta})]^{\alpha-\tilde \alpha\over
\alpha}[h(x)]^{\alpha'\over \alpha}.
$$
Using the Kac formula one can show that, for every $a>0$,  the
function $x\mapsto E_x e^{a\theta}$ is bounded on $D$ (this fact
is quite standard and hence we do not go into details here).
Therefore, the function $\phi$ is bounded on every compact
$S\subset D$.

Next, consider closed ball $E=\{x:\|x-x_0\|\leq r_0\}$ with the
boundary $S=\{x:\|x-x_0\|= r_0\}\subset D$,  and put
$\sigma=\inf\{t:X_t\in S\}$.

For $x\in E$, we have by the strong Markov property of $X$ that
$$
\phi(x)\leq E_x(e^{\tilde \alpha\sigma}\phi(X_\sigma))\leq (E_x
e^{\alpha\sigma})\sup_{y\in S}\phi(y).
$$
The function $x\mapsto E_x e^{\tilde \alpha\sigma}$ is bounded on
$E$ (again, we do not give a detailed discussion here). Hence
$\phi$ is bounded on $E$. Since $r_0$ can be taken arbitrarily
large, this means that that $\phi$ is locally bounded.

Now, let us  verify that condition 3) of Proposition \ref{p22}
holds true with $\alpha$ replaced by $\tilde \alpha$. We put
$\sigma^0=0$,
$$
\sigma^{2n-1}=\inf\{t\geq \sigma^{2n-2}:X_t\in S\},\quad
\sigma^{2n}=\inf\{t\geq \sigma^{2n-1}:X_t\in K\},\quad n\geq 1.
$$
For any  $a>0$, one has   $$ q\eqdef \max\left[\sup_{x\in
K}E_xe^{-a\tau_S}<1, \sup_{x\in S}E_xe^{-a\tau_K}<1\right]<1
$$
because $\mathrm{dist}\,(K,S)>0$ and $X$ is a  Feller process
with continuous trajectories. Therefore,
 \be\label{83}
E\Big[e^{-a(\sigma^{k+1}-\sigma^k)}\Big|\Ff_{\sigma^k}\Big]\leq
q\quad \hbox{a.s.,} \quad k\geq 0. \ee We have
$$
E_xe^{\tilde \alpha\tau_K^t}=\sum_{k=0}^\infty E_xe^{\tilde
\alpha\tau_K^t}\1_{\sigma^k\leq t<\sigma^{k+1}}, \quad x\in K.
$$
For $k$ even,  $X_t\in E$ a.s. on the set
$C_{k,t}\eqdef\{\sigma^k\leq t<\sigma^{k+1}\}$. In addition,
$C_{k,t}\in \Ff_t$. Hence
$$\ba
E_xe^{\tilde \alpha\tau_K^t}\1_{\sigma^k\leq t<\sigma^{k+1}}&=
E_x\left(\1_{\sigma^k\leq t<\sigma^{k+1}}E\Big[e^{\tilde
\alpha\tau_K^t}\Big|\Ff_t\Big]\right)= E_x\1_{\sigma^k\leq
t<\sigma^{k+1}}\phi(X_t)\\
&\leq \sup_{y\in E}\phi(y)\,P_x(\sigma^k\leq t<\sigma^{k+1}),
\quad k=2n. \ea
$$
For $k$ odd,  $\tau_K^t=\sigma^{k+1}-t\leq \sigma^{k+1}-\sigma^k$
a.s. on the set $C_{k,t}$. Hence
$$\ba E_xe^{\tilde \alpha\tau_K^t}\1_{\sigma^k\leq t<\sigma^{k+1}}&\leq
E_x\1_{\sigma^k\geq t}e^{\tilde
\alpha(\sigma^{k+1}-\sigma^k)}=E_x\left(\1_{\sigma^k\leq
t}E\Big[e^{\tilde
\alpha(\sigma^{k+1}-\sigma^k)}\Big|\Ff_{\sigma^k}\Big]\right)\\
&=E_x\1_{\sigma^k\leq t}\phi(X_{\sigma^k})\leq \sup_{y\in
E}\phi(y)\, P_x(\sigma^k\leq t).\ea
$$
Therefore,
$$
E_xe^{\tilde \alpha\tau_K^t}\leq \sup_{y\in
E}\phi(y)\sum_{k=0}^\infty P_x(\sigma^k\leq t), \quad x\in K.
$$
It follows from (\ref{83}) that $E_x e^{-a\sigma^k}\leq q^k, x\in
K$. Then
$$
P_x(\sigma^k\leq t)=P_x(-\sigma^k\geq -t)\leq e^{at}q^k, \quad
k\geq 0, x\in K,
$$
and consequently
$$
\sup_{x\in K, t\in [0,S]} E_xe^{\tilde \alpha\tau_K^t}\leq
e^{aS}(1-q)^{-1}\sup_{y\in E}\phi(y)<+\infty.
$$
We have verified that conditions 2), 3) of Proposition \ref{p22}
hold true with $\alpha$ replaced by $\tilde \alpha$. Also, we have
already seen that $X$ satisfies the extended Doeblin condition on
$K$. We complete the proof of the theorem applying Proposition
\ref{p22}.

\end{proof}

\begin{rem} The criterion given in Theorem \ref{t61} extends, in particular, the sufficient
condition from \cite{RW04}, Theorem 1.1. Indeed, under condition
(1.1) of the latter theorem one can verify that there exists a
function $\Phi:\Re\to \Re$ such that $\Phi(x)\to +\infty, |x|\to
\infty$ and the function $\phi=\Phi(\rho)$ satisfies the
Lyapunov-type condition (\ref{lyap}). This  yields existence of
an exponential $\phi$-coupling and hence the spectral gap
property.

On the other hand, it is worth to compare Theorem \ref{t61} with
the necessary and sufficient condition given in \cite{Mat97}. The
principal difference is that Theorem \ref{t61} deals with the
Poincar\'e inequality itself while in the part (ii) of the main
theorem in \cite{Mat97} some weak version of this inequality is
established. In addition,  sufficient condition of \cite{Mat97}
involves  the whole collection of hitting times $\{\tau_K: K$ is
closed and $\pi(K)\geq {1\over 2}\}$, while in Theorem \ref{t61}
condition $3'$) is imposed on one hitting time $\tau_K$, which
makes this theorem mush easier in application.
\end{rem}

\appendix\section{Proofs of Theorems \ref{tA1}, \ref{tA2} and Proposition \ref{p22}}

\subsection{Proof of Theorem \ref{tA1}}

Under conditions of Theorem \ref{tA1}, consider two independent
copies $Y^1,Y^2$ of the process $X$ with $Y^1_0=y^1,Y^2_0=y^2$
($y^1,y^2\in \XX$ are  arbitrary). It follows from condition 3)
that $\sup_{x\in K, t\in \ax}E_x\phi(X_t)<+\infty.$ In
particular, $\phi$ is bounded  on $K$. Denote
$$
D_1=\sup_{x\in K}\phi(x), \quad D_2=\sup_{x\in K, t\in
\ax}E_x\phi(X_t).
$$
 Take arbitrary
$\gamma\in (0,\alpha)$ and choose $c>D_1$ such that
$$
\delta\eqdef \sup_{x\in K, t\in\ax}
E_x\phi(X_t)\1_{\phi(X_t)>c}<1-{\gamma\over \alpha}.
$$
Define
$$K'\eqdef\{\phi\leq c\},\quad \theta\eqdef\inf\{t: Y_t^1\in K', Y_t^2\in K'\}.$$

\begin{lem}
\be\label{a1} E[\phi(Y_t^1)+\phi(Y^2_t)]\1_{\theta>t}\leq D_3
e^{-\gamma t}[\phi(y^1)+\phi(y^2)], \quad y^1,y^2\in \XX, \ee
$$
D_3=2D_2+3+2(D_2+1)^2\left(1-{\gamma\over
\alpha}-\delta\right)^{-1}.
$$
\end{lem}

\begin{proof}
We consider stopping time
$$\tau^1=\inf\{t:Y_t^1\in K\hbox{ or }Y_t^2\in K\},$$
and define the sequence of random variables $\iota_n,n\geq 1$
taking values in $\{1,2\}$ by $$
\begin{cases}\iota_n=n\, (\mathrm{mod}\,2),&\hbox{ if }Y^1_{\tau^1}\in
K\\
\iota_n=n+1\, (\mathrm{mod}\,2),& \hbox{ otherwise}
\end{cases}, \quad n\geq 1.
$$
Then, we  define define iteratively the sequence of stopping
times $$ \tau^{n+1}=\inf\{t>\tau^{n}:Y_t^{\iota_{n+1}}\in
K\},\quad n\geq 1.
$$ We put $\tau^0=0, \tau^\infty=\lim_n\tau^n$. Obviously,
$\tau^\infty=\inf\{t: Y_t^1\in K, Y_t^2\in K\}\geq \theta$.
Hence, \be\label{a2}
E[\phi(Y_t^1)+\phi(Y^2_t)]\1_{\theta>t}=\sum_{n=0}^\infty
E[\phi(Y_t^1)+\phi(Y^2_t)]\1_{\tau^n\leq t<\tau^{n+1},\theta>t}.
\ee
 Let us estimate separately summands in the right-hand side of
(\ref{a2}). Note that every process $Y^1,Y^2$ is strongly Markov,
and every  stopping time $\tau^n$, given the values $Y^1_{\tau^1}$
and $\tau^{n-1}$, is completely defined by the trajectory of one
component of the process $Y=(Y^1,Y^2)$. Because these components
are independent, this yields that $Y$ has strong Markov property
at every stopping time $\tau^n$.

We have $\tau^1=\tau_K^1\wedge \tau_K^2$, where $\tau_K^i$
denotes the hitting time for the process $Y^i, i=1,2$. Since
$Y^1,Y^2$ are independent,  we get from condition 2):
$$\ba
E[\phi(Y_t^1)+\phi(Y^2_t)]\1_{\tau^1>t,\theta>t}&\leq
E[\phi(Y_t^1)+\phi(Y^2_t)]\1_{\tau^1>t}\\
&\leq E_{y^1}\phi(X_t)\1_{\tau_K>t}+
E_{y^2}\phi(X_t)\1_{\tau_K>t}\leq e^{-\alpha
t}[\phi(y^1)+\phi(y^2)].\ea
$$
Next, consider the summand
$$ \ba E[\phi(Y_t^1)+\phi(Y^2_t)]\1_{\tau^1\leq t<\tau^{2},\theta>t}&\leq E[\phi(Y_t^1)+\phi(Y^2_t)]
(\1_{\tau_K^1\leq t,\tau^2_K>t}+\1_{\tau_K^2\leq t,\tau^1_K>t})\\
& = \Big(E_{y^1} \phi(X_t)\1_{\tau_K\leq t}\Big)P_{y^2}(\tau_K>t)+
\Big(E_{y^1} \phi(X_t)\1_{\tau_K> t}\Big)P_{y^2}(\tau_K\leq t)\\
&+\Big(E_{y^2} \phi(X_t)\1_{\tau_K\leq t}\Big)P_{y^1}(\tau_K>t)+
\Big(E_{y^2} \phi(X_t)\1_{\tau_K> t}\Big)P_{y^1}(\tau_K\leq t).\ea
$$
Recall that $\phi\geq 1$. Then condition 2) yields
$$
P_y(\tau_K>t)\leq e^{-\alpha t}\phi(y).
$$
By the strong Markov property of $X$, we have
$$
E_x\phi(X_t)\1_{\tau_K\leq t}=E_x\left[
E_y\phi(X_{t-s})\Big|_{s=\tau_K, y=X_{\tau_k}}\right]\leq D_2.
$$
Therefore,
$$E[\phi(Y_t^1)+\phi(Y^2_t)]\1_{\tau^1\leq t<\tau^{2},\theta>t}\leq 2(D_2+1)[\phi(y^1)+\phi(y^2)]e^{-\alpha t}.
$$
Remark that, in fact, we have proved inequality
$$
E[\phi(Y_t^1)+\phi(Y^2_t)]\1_{t<\tau^{2}}\leq
2(D_2+1)[\phi(y^1)+\phi(y^2)] e^{-\alpha t}$$ which yields
\be\label{a21}Ee^{\gamma\tau^{2}}\leq 2(D_2+1)\left({\alpha\over
\alpha-\gamma}\right)[\phi(y^1)+\phi(y^2)].\ee

We have estimated two first summands in (\ref{a2}). The other
summands  can be estimated iteratively in the following way. We
have \be\label{a3}\ba E\phi(Y_t^{\iota_{n+1}})\1_{\tau^n\leq
t<\tau^{n+1},\theta>t}&\leq E\phi(Y_t^{\iota_{n+1}})\1_{\tau^n\leq
t<\tau^{n+1},\theta>\tau_n}\\
&=e^{-\gamma t}E \left(e^{\gamma\tau^{n}}\1_{\tau^n\leq
t,\theta>\tau_n}E\Big[\phi(Y_t^{\iota_{n+1}})e^{\gamma
(t-\tau^{n})}\1_{t<\tau^{n+1}}\Big|\Ff_{\tau^n}\Big]\right)\\
&\leq  e^{-\gamma t} Ee^{\gamma\tau^{n}}\1_{\tau^n\leq
t,\theta>\tau_n}\phi(Y^{\iota_{n+1}}_{\tau^n}).\ea \ee Here,
$\{\Ff_t\}$ denotes the natural  filtration for $Y$. We have used
strong Markov property at the point $\tau^n$ and inequality
$$
E_xe^{\gamma t}\phi(X_t)\1_{\tau_K>t}\leq \phi(x)
$$
that follows from 2).

Next, the processes $U^n_t\eqdef Y^{\iota_{n}}_{t-\tau^n},
V_t^n\eqdef Y^{\iota_{n+1}}_{t-\tau^{n}}$ are conditionally
independent w.r.t.  $\Ff_{\tau^n}$. Denote $\varsigma^n$ the
first time for $V^n$ to hit $K$. Then
$\tau^{n+1}=\varsigma^n+\tau^n$.

We have
$$\ba
E\phi(Y_t^{\iota_{n}})\1_{\tau^n\leq t<\tau^{n+1},\theta>t}&\leq
E\phi(Y_t^{\iota_{n}})\1_{\tau^n\leq t<\tau^{n+1},\theta>\tau^n}\\
&=E\left(E\Big[\phi(U^n_{t-\tau_n})\Big|\Ff_{\tau^n}\Big]E\Big[\1_{\varsigma^n>t-\tau_n}\Big|\Ff_{\tau^n}\Big]\right)\1_{\tau^n\leq
t,\theta>\tau^n}\\
&\leq D_2
E\left(E\Big[\1_{\varsigma^n>t-\tau_n}\Big|\Ff_{\tau^n}\Big]\right)\1_{\tau^n\leq
t,\theta>\tau^n}.\ea
$$
In the last inequality we have used that, by the construction,
$U^n_0=Y^{\iota_{n}}_{\tau^n}\in K$, and hence
$$
E\Big[\phi(U^n_{t-\tau_n})\Big|\Ff_{\tau^n}\Big]\1_{t\geq
\tau_n}\leq\sup_{x\in K,t\in\ax}E\phi(X_t)=D_2.
$$
 Then,
since $\phi\geq 1$, \be\label{a4}\ba
E\phi(Y_t^{\iota_{n}})\1_{\tau^n\leq t<\tau^{n+1},\theta>t}&\leq
D_2 e^{-\gamma t}E \left(e^{\gamma\tau^{n}}\1_{\tau^n\leq
t,\theta>\tau_n}E\Big[\phi(Y_t^{\iota_{n+1}})e^{\gamma
(t-\tau^{n})}\1_{t<\tau^{n+1}}\Big|\Ff_{\tau^n}\Big]\right)\\
&\leq   D_2 e^{-\gamma t} Ee^{\gamma\tau^{n}}\1_{\tau^n\leq
t,\theta>\tau_n}\phi(Y^{\iota_{n+1}}_{\tau^n})\leq D_2
Ee^{\gamma\tau^{n}}\1_{\theta>\tau_n}\phi(Y^{\iota_{n+1}}_{\tau^n}).\ea
\ee

Let us estimate
$$
Ee^{\gamma\tau^{n}}\1_{\theta>\tau_n}\phi(Y^{\iota_{n+1}}_{\tau^n}).
$$

We have $Y^{\iota_n}_{\tau^n}\in K$, and hence inequality
$\theta>\tau^n$ implies that $Y^{\iota_{n+1}}_{\tau^n}\not \in
K'$. Recall that $\phi>c$ outside $K'$, and $c$ is chosen in such
a way that $E_x\phi(X_t)\1_{\phi(X_t)>c}<\delta$ for any $x\in K,
t\in\ax$. Therefore, the same arguments with those that lead to
(\ref{a4}) provide
$$
Ee^{\gamma\tau^{n}}\1_{\theta>\tau_n}\phi(Y^{\iota_{n+1}}_{\tau^n})\leq\delta
Ee^{\gamma\tau^{n-1}}\1_{\theta>\tau_{n-1}} E\Big[e^{\gamma
(\tau^{n}-\tau^{n-1})}\Big|\Ff_{\tau^{n-1}}\Big].
$$
It can be verified easily that, under condition 2),
$$
E_xe^{\gamma\tau_K}\leq {\alpha\over \alpha-\gamma} \phi(x),
\quad x\in \XX.
$$
Hence, for $n\geq 2$,
$$\ba
Ee^{\gamma\tau^{n}}\1_{\theta>\tau_n}\phi(Y^{\iota_{n+1}}_{\tau^n})&\leq{\delta\alpha\over
\alpha -\gamma}
Ee^{\gamma\tau^{n-1}}\1_{\theta>\tau_{n-1}}\phi(Y^{\iota_{n}}_{\tau^{n-1}})\leq
\cdots\\
&\leq \left({\delta\alpha\over \alpha
-\gamma}\right)^{n-2}Ee^{\gamma\tau^{2}}\phi(Y^{\iota_{3}}_{\tau^2})\leq
D_2\left({\delta\alpha\over \alpha
-\gamma}\right)^{n-2}Ee^{\gamma\tau^{2}}.\ea
$$
 The latter
estimate and (\ref{a21}) provide
$$
E[\phi(Y_t^1)+\phi(Y_t^2)]\1_{\tau^n\leq t<\tau^{n+1},\theta>t}
\leq 2e^{-\gamma t}(D_2+1)^2\left({\alpha\over \alpha
-\gamma}\right)\left({\delta\alpha\over \alpha
-\gamma}\right)^{n}[\phi(y^1)+\phi(y^2)], \quad n\geq 2.
$$
This inequality, together with (\ref{a3}) and (\ref{a4}), gives
(\ref{a1}) after summation by $n$.
\end{proof}

 The rest of the proof of Theorem \ref{tA1} is based on the construction described in \cite{Kul09}, Section
 3.2. Here, we give the sketch of the construction, referring interested reader to \cite{Kul09} for details,
  discussion and references.

 Consider two types of "elementary couplings": a "simple coupling"
 and a "gluing coupling".
 The simple coupling is just a two-component
 Markov process $Z=(Z^1,Z^2)$ such that either $Z^1, Z^2$ are
 independent if $Z^1_0=z^1, Z^2_0=z^2, z^{1,2}\in \XX,$ and
 $z^1\not=z^2$, or $Z^1=Z^2$ if $Z^1_0=Z_0^2=z\in\XX$. The gluing
 coupling  is constructed on a given time interval $[0,T]$ for fixed $z^1,z^2\in \XX$  in such  a
 way that
 $Z^1_0=z^1, Z^2_0=z^2$, and
 $$
 P(Z_T^1=Z_T^2)=1-{1\over
 2}\|P_T(z^1,\cdot)-P_T(z^2,\cdot)\|_{var}.
 $$

Next, we construct the "switching coupling" $Z$ as an appropriate
mixture of these elementary ones. Namely, for a given $z^1, z^2\in
\XX$ we consider a simple coupling $Z^{s}=(Z^{s,1},Z^{s,2})$ with
$Z^{s,1}_0=z^1, Z^{s,2}_0=z^2$ and define $\theta^1=\min\{t:
Z_t^s\in K'\times K'\}$ (the set $K'$ is defined above). Then the
value of $Z^s$ at the random time moment $\theta^1$ is
substituted, as  the starting position, into an independent copy
of the gluing coupling $Z^{g}$. The switching coupling $Z$ is
defined, up to the random moment of time $\theta^2=\theta^1+T$, as
$$
Z_t=\begin{cases}Z^s_t,&t\leq \theta_1,\\
Z^g_{t-\theta^1}, t\in(\theta^1,\theta^2].
\end{cases}
$$
Then this construction is iterated: the value $Z_{\theta^2}$ is
substituted, as the starting position, into an independent copy of
the simple coupling, etc. This construction gives a coupling $Z$
and a sequence of stopping times $\theta^k, k\geq 1$ such that

(a) if $Z^1_{\theta^k}=Z^2_{\theta^k}$ for some $k$, then
$Z^1_{t}=Z^2_{t}$ for $t>\theta_k$;

(b) for every $k$,
$$P(Z^1_{\theta^{2k}}\not= Z^2_{\theta^{2k}}|\Ff_{\theta^{2k-1}})\leq
\kap(T,K')\quad \hbox{ a.s.,}$$ where $\{\Ff_t\}$ denotes the
natural filtration for $Z$.

 Recall that $\phi(x)\to \infty, x\to
\infty$, hence $K'=\{\phi\leq c\}$ has a compact closure.
Therefore, by the local Doeblin condition, $T$ can be chosen in
such a way that $\kap(T,K')\leq \kap(T,\mathrm{closure}(K'))<1$.

Let us estimate the value
$$
E\Big[\phi(Z^1_t)+\phi(Z_t^2)\Big]\1_{Z_t^1\not=Z_t^2}.
$$
Property (a) allows one to write \be\label{a5}\ba
E\Big[\phi(Z^1_t)+\phi(Z_t^2)\Big]\1_{Z_t^1\not=Z_t^2}&\leq
E\Big[\phi(Z^1_t)+ \phi(Z_t^2)\Big]\1_{\theta^2>t}\\
&+\sum_{k=1}^\infty
E\Big[\phi(Z^1_t)+\phi(Z_t^2)\Big]\1_{\theta^{2k}\leq
t<\theta^{2k+2}}\1_{Z^1_{\theta^{2k}}\not=Z^2_{\theta^{2k}}}. \ea
\ee Take arbitrary $\beta\in(0,\alpha)$. It follows immediately
from (\ref{a1}) with $\gamma=\beta$ that the first summand in the
right hand side of (\ref{a5}) is estimated by $Ce^{-\beta
t}[\phi(z^1)+\phi(z^2)]$. The same inequality yields that
$$\ba
E\left(\Big[\phi(Z^1_t)+\phi(Z_t^2)\Big]\1_{t<\theta^{2k+2}}|\Ff_{\theta^{2k}}\right)&=
E\left(\Big[\phi(Z^1_t)+\phi(Z_t^2)\Big]\1_{t-T<\theta^{2k+1}}|\Ff_{\theta^{2k}}\right)\\
& \leq
C\Big[\phi(Z^1_{\theta^{2k}})+\phi(Z_{\theta^{2k}}^2)\Big]e^{\beta(\theta^{2k}+T-t)}.\ea
$$
Hence, we can estimate the $k$-th summand in the sum in the right
hand side of (\ref{a5}) by
$$
Ce^{-\beta t}e^{\beta
T}E\Big[\phi(Z^1_{\theta^{2k}})+\phi(Z_{\theta^{2k}}^2)\Big]e^{\beta\theta^{2k}}\1_{Z^1_{\theta^{2k}}\not=Z^2_{\theta^{2k}}}.
$$
Next, we remove the function $\phi$ from this estimate:
$$\ba
E\Big[\phi(Z^1_{\theta^{2k}})+\phi(Z_{\theta^{2k}}^2)\Big]e^{\beta\theta^{2k}}\1_{Z^1_{\theta^{2k}}\not=Z^2_{\theta^{2k}}}
&\leq
E\Big[\phi(Z^1_{\theta^{2k}})+\phi(Z_{\theta^{2k}}^2)\Big]e^{\beta\theta^{2k-1}+\beta
T}\1_{Z^1_{\theta^{2k-2}}\not=Z^2_{\theta^{2k-2}}}\\
& =Ee^{\beta\theta^{2k-1}+\beta
T}\1_{Z^1_{\theta^{2k-2}}\not=Z^2_{\theta^{2k-2}}}
E\Big[\phi(Z^1_{\theta^{2k-1}+T})+\phi(Z_{\theta^{2k-1}+T}^2)|\Ff_{\theta^{2k-1}}\Big]\\
&\leq Ce^{\beta T}Ee^{\beta\theta^{2k-1}+\beta
T}\1_{Z^1_{\theta^{2k-2}}\not=Z^2_{\theta^{2k-2}}}.\ea
$$
Here, we have used condition 3) and notation $\theta^0=0$ (recall
that $\phi(Z^1_{\theta^{2k-1}})\leq c,
\phi(Z^1_{\theta^{2k-1}})\leq c$ by the construction of the
coupling $Z$). Hence, (\ref{a5}) can be rewritten as
$$
E\Big[\phi(Z^1_t)+\phi(Z_t^2)\Big]\1_{Z_t^1\not=Z_t^2}\leq
Ce^{-\beta t}\Big[1+\sum_{k=1}^\infty
Ee^{\beta\theta^{2k-1}}\1_{Z^1_{\theta^{2k-2}}\not=Z^2_{\theta^{2k-2}}}\Big].
$$
Next, from the property (b) of the coupling $Z$, we have
$$
Ee^{\beta\theta^{2k-1}}\1_{Z^1_{\theta^{2k-2}}\not=Z^2_{\theta^{2k-2}}}\leq
\left[Ee^{2\beta\theta^{2k-1}}\right]^{1\over 2}P^{1\over
2}(Z^1_{\theta^{2k-2}}\not=Z^2_{\theta^{2k-2}}) \leq
\left[Ee^{2\beta\theta^{2k-1}}\right]^{1\over 2}\kap^{k-1\over
2}(T,K').$$

Up to this moment, $\beta\in (0,\alpha)$ was taken in an
arbitrary way. On the other hand, (\ref{a1}) yields that, for
fixed $\gamma<\alpha$ and $k>1$,
$$
E[\1_{\theta^{2k-1}-\theta^{2k-3}>t}|\Ff_{\theta^{2k-3}}]=E[\1_{\theta^{2k-2}-\theta^{2k-3}>t-T}|\Ff_{\theta^{2k-3}}]\leq
C e^{-\gamma t}.
$$
Hence, for every $q>1$,  one can take $\beta>0$ small enough for
$E[e^{2\beta(\theta^{2k-1}-\theta^{2k-3})}|\Ff_{\theta^{2k-3}}]\leq
q$ a.s. At last, by (\ref{a1}),
$$
Ee^{2\beta\theta^1}\leq C[\phi(z^1)+\phi(z^2)]
$$
for $\beta<{\alpha\over 2}$. This, finally, provides the estimate
\be\label{a7}
E\Big[\phi(Z^1_t)+\phi(Z_t^2)\Big]\1_{Z_t^1\not=Z_t^2}\leq
Ce^{-\beta t}[\phi(z^1)+\phi(z^2)]\Big[1+\sum_{k=1}^\infty
\left(q\kap(T,K')\right)^{k-1\over 2}\Big] =C'e^{-\beta
t}[\phi(z^1)+\phi(z^2)], \ee where $C'=C\Big[1+\sum_{k=1}^\infty
\left(q\kap(T,K')\right)^{k-1\over 2}\Big]$. Note that
$C'<+\infty$ if, in the construction described before, $q>1$ is
taken in such a way that
$$
q\kap(T,K')<1.
$$
Now, we can put $z^1=x$ and assume $z^2$ to be random and have
its distribution equal to $\pi$. Then $Z$ is a
$(\delta_x,\pi)$-coupling, and, by (\ref{a7}),
$$
E\Big[\phi(Z^1_t)+\phi(Z_t^2)\Big]\1_{Z_t^1\not=Z_t^2}\leq
C'e^{-\beta t}\left[\phi(x)+\int_\XX\phi\,d\pi\right]\leq
Ce^{-\beta t}\phi(x),
$$
here we took into account that $\phi\geq 1$ and $\int_\XX\phi\,
d\pi<+\infty$. The proof of Theorem \ref{tA1} is complete.

\begin{rem}\label{rA1} Condition 3) of Theorem \ref{tA1} yields  $\sup_{x\in K, t\in
\ax}E_x\phi(X_t)<+\infty.$ On the other hand, existence of
exponential $\phi$-coupling provides that $P_t(x,dy)\to \pi(dy),
t\to \infty$ in variation for every $x\in K$. Consequently, under
conditions of Theorem \ref{tA1}, $\int_\XX\phi\,d\pi<+\infty$.
One can easily deduce similar statement under conditions of
Theorem \ref{tA2} and Proposition \ref{p22}.
\end{rem}

\subsection{Proof of Theorem \ref{tA2}}

One can see that, in the previous arguments,  the only place
where condition 1) of Theorem \ref{tA1} was used is that the set
$K'=\{\phi\leq c\}$ has a compact closure. This property was not
required straightforwardly: we use it only to verify that
$\kap(T, K')<1$, i.e. that $X$ satisfies the Doeblin condition on
$K'$.  Hence, literally the same arguments ensure that the process
$X$ admits an exponential $\phi$-coupling assuming that $X$
satisfies the Doeblin condition on every set of the type
$\{\phi\leq c\}$. Therefore, the following statement yields
Theorem  \ref{tA2}.
\begin{lem} Assume that conditions 2),3) of
Theorem \ref{tA1} hold true and  $X$ satisfies the extended
Doeblin condition on $K$.

Then $X$ satisfies the Doeblin condition on every set of the type
$\{\phi\leq c\}$.
\end{lem}

\begin{proof} We use an auxiliary construction of the {\it extended gluing coupling}. This coupling is defined,
  for fixed $z^1,z^2\in \XX$, $t_1,t_2\in \Re$,  in such  a
 way that
 $Z^1_0=z^1, Z^2_0=z^2$, and
 $$
 P(Z_{t-1}^1\not=Z_{t_2}^2)=1-{1\over
 2}\|P_{t_1}(z^1,\cdot)-P_{t_2}(z^2,\cdot)\|_{var}.
 $$
One can construct this coupling using literally the same
arguments with those used in the construction of the (usual)
gluing coupling (see \cite{Kul09}, Section
 3.2), with
the terminal time moment $T$ replaced by $t_1$ for the component
$Z^1$ and $t_2$ for the component $Z^2$. It can be verified that
such a construction can be made in a joinly measurable way w.r.t.
probability variable and $z^{1,2}, t^{1,2}$ (we refer for a more
detailed discussion of the measurability problems to \cite{Kul09},
Section
 3.2).

 Under condition 2) of Theorem \ref{tA1},
$$
P_x(\tau_K>t)\leq e^{-\alpha t}\phi(x).
$$
Therefore, for $Q\in\ax$ large enough,
$$
P_x(\tau_K\leq Q)\geq {1\over 2},\quad x\in K'=\{\phi\leq c\}.
$$
Consider two independent copies $Y^1,Y^2$ of the process $X$
starting from the points $x^1,x^2\in K'$. Denote
$$
\tau^{1,2}=\inf\{t\geq 0: Y^{1,2}_t\in K\}.
$$
Since $P(\tau^1\leq Q,\tau^2\leq Q)\geq {1\over 4}$, one of the
following inequalities hold:
$$
P(\tau^1\leq \tau^2\leq Q)\geq {1\over 8}, \quad P(\tau^2\leq
\tau^1\leq Q)\geq {1\over 8}.
$$
Assume that the first inequality holds (this does not restrict
generality). Then we put $T=Q+T_1$ (here $T_1$ comes from
(\ref{81})) and construct the coupling $Z_t, t\in [0, T]$ in the
following way. If inequality $\tau^1\leq \tau^2\leq Q$ does not
hold, then $Z^{1,2}=Y^{1,2}$. Otherwise we consider an
independent copy of the extended gluing coupling, and substitute
in it $Z^1_{\tau^1}, Z^2_{\tau^2}$  instead of the initial values
$z^1, z^2$,  and $T-\tau^1$, $T-\tau^2$ instead of the terminal
time moments $t^1, t^2$. Under such a construction,
$$
P(Z_T^1=Z_T^2)\geq {1\over 8}\Big(1-{1\over 2}\sup_{z^1,z^2\in K,
t^1,t^2\in [T_1,
T_1+Q]}\|P_{t_1}(z^1,\cdot)-P_{t_2}(z^2,\cdot)\|_{var}\Big).
$$
Therefore, \be\label{82} 1-\kap(T, K')\geq {1\over
8}\Big(1-\kap(T_1, T_1+Q, K)\Big). \ee Clearly, $\kap(T_1, T_2',
K)\leq \kap(T_1, T_2, K)$ for every $T_2'\in [T_1, T_2]$. On the
other hand,  using Chapman-Kolmogorov equation, one can verify
easily that inequality (\ref{81}) implies the same inequality
with $T_2$ replaced by arbitrary $T_2'>T_2$. Hence, under
condition (\ref{81}), we can put $T_2'=T_1+Q$ and get $\kap(T_1,
T_1+Q, K)<1$. This, together with (\ref{82}), provides that
$\kap(T,K')<1$.
\end{proof}
\subsection{Proof of Proposition \ref{p22}} It can be verified
easily that condition 1) of Proposition \ref{p22} implies that
$\phi(x)=E_xe^{\alpha'\tau_K}$ satisfies condition 1) of Theorem
\ref{tA1}. By the Markov property of $X$,
$$\phi(X_t)=\Big[E_ye^{\alpha'\tau_K}\Big]_{y=X_t}=E\Big[e^{\alpha'\tau^t_K}\Big|\Ff_t\Big],\quad t\geq 0$$
(see Section 2.1 for the  notation $\tau^t_K$). We have
$\tau^t_K=\tau_K-t$ on the set $\{\tau_K>t\}$. Therefore,
$$
E_x\phi(X_t)\1_{\tau_K>t}=E_xe^{\alpha'\tau_K^t}\1_{\tau_K>t}=E_xe^{\alpha'(\tau_K-t)}\1_{\tau_K>t}\leq
e^{-\alpha't}\phi(x).
$$
Hence, condition 2) of Theorem \ref{tA1} holds true with $\alpha$
replaced by $\alpha'$.

Condition 3) of Theorem \ref{tA1}, in fact, is the claim for the
function $\phi$ to be uniformly integrable w.r.t. the family of
distributions $\{P_t(x,\cdot), x\in K, t\in \ax\}$. Clearly, it
is satisfied if
$$
\sup_{x\in K, t\in \ax} E_x\phi^r(X_t)<+\infty.
$$
for some $r>1$. Therefore, for the function $\phi(x)=E_xe^{\alpha'
\tau_K}$, condition 3) of Theorem \ref{tA1} holds true provided
that
$$
\sup_{x\in K, t\in \ax} E_x\phi^{\alpha/\alpha'}(X_t)<+\infty.
$$
(recall that $\alpha'\in (0,\alpha)$). By the H\"older inequality,
$$\phi^{\alpha/\alpha'}(y)\leq E_ye^{\alpha\tau_K}.
$$
 Therefore, Proposition
\ref{p22} is provided by Theorems \ref{tA1}, \ref{tA2} and the
following statement.

\begin{lem} Let function $\psi:\XX\to [1,+\infty)$ be such that
$$E_x\psi(X_t)\1_{\tau_K>t}\leq e^{-\alpha t}\psi(x),x\in\XX;$$
$$\exists\, S>0:\, \sup_{x\in K, t\leq S} E_x
\psi(X_t)<+\infty.
$$
Then
$$
\sup_{x\in K, t\in \ax} E\psi(X_t)<+\infty.
$$
\end{lem}
\begin{proof} For $t>S$, one has
\be\label{a8}\ba
E_x\psi(X_t)&=E_x\psi(X_t)\1_{\tau^S_K>t}+E_x\psi(X_t)\1_{\tau^S_K\leq
t}\leq \int_\XX \Big[E_y\psi(X_t)\1_{\tau_K>t-S}\Big]P_S(x,dy)\\
&+E_x\psi(X_t)\1_{\tau^S_K\leq t}\leq e^{-\alpha t+\alpha
S}E\psi(X_S)+E_x\psi(X_t)\1_{\tau^S_K\leq t}. \ea \ee
 Denote
$\TT_k=[kS,(k+1)S]$. It follows from (\ref{a8}) that
$$\sup_{t\in T_k}E_x\psi(X_t)\leq
e^{-\alpha(k-1)S}E_x\psi(X_S)+ \sup_{t\in
T_{k-1}}E_x\psi(X_t),\quad  k\geq 1,
$$
and, consequently,
$$
\sup_{t\in T_k}E_x\psi(X_t)\leq
(e^{-\alpha(k-1)S}+\dots+1)E_x\psi(X_S)+\sup_{t\leq S}E_x\psi(X_t)
 \leq \left[1+(1-e^{-\alpha
S})^{-1}\right]\sup_{t\leq S}E_x\psi(X_t).
$$
\end{proof}


\begin{thebibliography}{99}

\bibitem[AF]{AF} \textit{Aldous, F. and  Fill, J.} Reversible Markov chains and random walks on
graphs,\\ http://www.stat.berkeley.edu/users/aldous/RWG/book.html.

\bibitem[And91]{And91} \textit{Anderson, W. J.} (1991) Continuous-Time Markov Chains, Springer
Series in Statistics, Springer, Berlin.

\bibitem[Chen00]{Chen00}\textit{Chen, M.-F.} (2000) Equivalence of exponential ergodicity and
$L_2$-exponential convergence for Markov chains, Stoch. Proc.
Appl. 87, 2, 281 -– 297.

\bibitem[DFG09]{DFG09} \textit{Douc, R., Fort, G, and Guillin, A.} (2009)
 Subgeometric rates of convergence of $f$-ergodic strong Markov processes, Stochastic Processes and Appl.
  119, 3, 897 -- 923.

\bibitem[IKO62]{IKO62}\textit{Il'in, A.M., Kalashnikov, A.S. and Oleinik O.A.} (1962) Linear second order
parabolic equations, Uspekhi Mat. Nauk 17, 3, 3 –- 143.

\bibitem[KK09]{KK} \textit{Knopova, V.P. and  Kulik, A.M.} (2009)
 Exact asymptotics for a distribution density of  certain L\'evy
functionals, arXiv:0911.4683

\bibitem[KS81]{KS81} \textit{Krylov, N.V. and  Safonov, M.V.}(1981) A certain property of solutions
of parabolic equations with measurable coefficients, Math. USSR
Izvestija 16,  151 -- 164.


\bibitem[Kul06]{K06}  \textit{Kulik, A.M.} (2006) Stochastic calculus of variations for general L\'evy processes
and its applications to jump-type SDE's  with non-degenerated
drift, arxiv.org:math.PR/0606427v2.

\bibitem[Kul08]{Kul08} \textit{Kulik, A.M.} (2008) Absolute continuity and convergence in variation for
distributions of a functionals of Poisson point measure,
arXiv:0803.2389


\bibitem[Kul09]{Kul09} \textit{Kulik, A.M.} (2009) Exponential ergodicity of the solutions to SDE's with a jump
noise, Stochastic Processes and Appl.  119, 2, 602 -– 632.

\bibitem[MR92]{MR92} \textit{Ma, Z.-M. and R\"ockner, M.} (1992) Introduction to the theory
of (non-symmetric) Dirichlet forms,  Springer-Verlag, London,
Ltd., London.


\bibitem[Mas07]{Mas07}\textit{Masuda, H.} (2007) Ergodicity and exponential $\beta$-mixing bounds for
multidimensional diffusions with jumps, Stoch. Proc. and Appl.,
117  35 -– 56.

\bibitem[Mat97]{Mat97} \textit{Mathieu, P.} (1997) Hitting times and
spectral gap inequalities, Ann. Inst. Henri Poincar\'e 33, 4, 437
-- 465.

\bibitem[MT93]{MT93} \textit{Meyn, S.P. and  Tweedie, R.L.} (1993) Markov chains and stochastic stability,
 Springer-Verlag London, Ltd., London.

\bibitem[Nag86]{Nag86} \textit{Nagel, R. (ed)} (1986) One-parameter semigroups of positive
operators, Lecture Notes Math 1184, Springer-Verlag.

\bibitem[RR97]{RR97} \textit{Roberts, G.O. and Rosenthal, J.S.} (1997) Geometric ergodicity
and hybrid Markov chains, Electron. Comm. Probab. 2, 13 -- 25.

\bibitem[RW04]{RW04} \textit{R\"ockner, M. and  Wang, F.-Y.} (2004) Spectrum for a class of (nonsymmetric)
diffusion operators, Bull. London Math. Soc. 36, 95 -- 104.

\bibitem[Rud73]{Rud73}\textit{Rudin, W.} (1973) Functional
Analysis, McGraw-Hill, New-York.


\bibitem[SY84]{SY84}\textit{Sato, K. and Yamazato, M.} (1984)
Operator-self-decomposable distributions as limit distributions
of processes of Ornstein-Uhlenbeck type,  Stochastic Process.
Appl., 17, 73 –- 100.

\bibitem[Ver87]{Ver87} \textit{Veretennikov, A.Yu.} (1987) On estimates
of mixing rate for stochastic equations,  Theory of Prob. and
Appl. 32, 299 -- 308 (in Russian).

\bibitem[Ver99]{Ver99}\textit{Veretennikov, A.Yu.} (1999) On polynomial
mixing and rate of convergence for stochastic differential and
difference equations, Theory of Prob. and Appl. 44,  312 -- 327
(in Russian).

\bibitem[Wang00]{Wang00} \textit{Wang, F.-Y.} (2000) Functional inequalities, semigroup properties and spectrum
estimates, Infin. Dimens. Anal., Quant. Probab. and Related Topics
3, 2, 263 -– 295.
 \end{thebibliography}
\end{document}